\documentclass{amsart}
\usepackage[T1]{fontenc}
\usepackage[utf8]{inputenc}
\usepackage{lmodern,microtype}
\usepackage{amsmath,amssymb,mathtools,mathrsfs,enumitem}
\usepackage{xcolor}
\usepackage[colorlinks=true,linkcolor=blue,citecolor=blue,urlcolor=blue]{hyperref}
\usepackage{tikz}
\usetikzlibrary{arrows.meta}
\newtheorem{theorem}{Theorem}[section]
\newtheorem{lemma}[theorem]{Lemma}
\newtheorem{proposition}[theorem]{Proposition}
\newtheorem{corollary}[theorem]{Corollary}
\newtheorem{maintheorem}{Theorem}

\newtheorem*{theorem*}{Theorem}
\theoremstyle{remark}
\newtheorem{remark}[theorem]{Remark}
\numberwithin{equation}{section}

\newcommand{\Prob}{\operatorname{Prob}}
\newcommand{\dom}{\operatorname{dom}}
\newcommand{\ran}{\operatorname{ran}}
\newcommand{\id}{\operatorname{id}}

\newcommand{\Inf}{\operatorname{Inf}}
\newcommand{\Aff}{\operatorname{Aff}}
\newcommand{\Homeo}{\operatorname{Homeo}}
\newcommand{\Clop}{\operatorname{Clop}}
\newcommand{\cM}{\mathcal M}
\newcommand{\cG}{\mathcal G}
\newcommand{\one}{\mathbf 1}
\newcommand{\acts}{\curvearrowright}
\newcommand{\precsimG}{\precsim_{\Gamma}}
\newcommand{\Talpha}{T(\alpha)}
\newcommand{\Halpha}{H(\alpha)}
\newcommand{\diam}{\operatorname{diam}}

\title[Comparison and almost finiteness]{
Comparison and Almost Finiteness for Actions of Amenable Groups}
\author[E. Glasner]{Eli Glasner}
\address[E. Glasner]{Department of Mathematics, Tel Aviv University, Tel Aviv, Israel}

\email{glasner@math.tau.ac.il}
\author[C. Liu]{Chunlin Liu\textsuperscript{*}}
\address[C. Liu]{School of Mathematical Sciences, Dalian University of Technology, Dalian, 116024, P.R. China and Institute of Mathematics, Polish Academy of Sciences, ul. Śniadeckich 8, 00-656 Warszawa, Poland}

\email{chunlinliu@mail.ustc.edu.cn; chunlin.liu.math@gmail.com}
\thanks{\textsuperscript{*}Corresponding author: Chunlin Liu,
\href{mailto:chunlinliu@mail.ustc.edu.cn}
{chunlinliu@mail.ustc.edu.cn}.}
\date{}

\begin{document}
\begin{abstract}
We prove that every  action of a countably infinite
discrete amenable group on a nonempty compact Hausdorff zero-dimensional space  has dynamical comparison, without assuming
minimality, freeness, or
metrizability.  More precisely, strict inequalities under all invariant
probability measures imply comparison in the clopen type semigroup with an order-unit remainder. For minimal Cantor actions, the clopen type semigroup is cancellative and almost unperforated, and its
canonical map onto the positive cone of the coinvariant group is an
isomorphism of ordered monoids.  These results answer Melleray's
cancellation question and the amenable case of his comparison
question.

For minimal Cantor actions, we also prove goodness of the associated affine evaluation map and obtain a minimal homeomorphism of the same Cantor space with the same invariant probability measures.
 We identify the precise obstruction to
recovering clopen equidecomposability from affine evaluation alone.
This obstruction is given by the infinitesimal subgroup: for every
nonempty proper clopen set $A$, it parametrizes the
topological-full-group orbits of clopen sets having the same affine
evaluation as $A$.

Combining comparison with the theorems of Kerr and Szab\'o,
we prove that every free action of a countably infinite discrete
amenable group on a nonempty compact metrizable space with the
topological small boundary property is almost finite.
This settles Naryshkin’s almost-finiteness conjecture for finite-dimensional compact metrizable spaces.
If the action is also minimal, its reduced crossed product is
$\mathcal Z$-stable, has nuclear dimension at most one, and satisfies
all the regularity conditions in the Toms--Winter conjecture.

\end{abstract}
\maketitle

\tableofcontents

\section{Introduction}\label{sec:introduction}
For minimal homeomorphisms of a Cantor space, clopen partitions provide a fundamental link between orbit structure, invariant measures, and ordered groups. This connection underlies both the Bratteli--Vershik theory of Herman--Putnam--Skau \cite{HermanPutnamSkau} and the orbit-equivalence theory of Giordano--Putnam--Skau \cite{GiordanoPutnamSkau}. 
The following related 
comparison
principle was established by Glasner and Weiss \cite[Lemma 2.5]{GlasnerWeiss}.
\begin{theorem*}[Glasner--Weiss]
Let $(X,T)$ be a Cantor minimal system, and let $A,B\subseteq X$
be clopen sets satisfying
 $\mu(A)<\mu(B)$
 for every $T$-invariant Borel probability measure $\mu.$
 Then there exists $g\in[[T]]$ such that
\[
   gA\subseteq B.
\]
Here $[[T]]$ is the topological full group of $T$, whose elements
are homeomorphisms of $X$ which agree with powers of $T$ on the
parts of a finite clopen partition.
\end{theorem*}

The main purpose of this paper is to establish this comparison
principle for actions of arbitrary countably infinite discrete
amenable groups on compact zero-dimensional spaces. Neither
minimality nor freeness is required for dynamical comparison. The
additional structure available in the minimal Cantor case allows us
to identify clopen types with the positive cone of the coinvariant
group and to classify exact equidecompositions. We first isolate the
finite clopen comparison property underlying the full-group
conclusion of Glasner--Weiss.

More precisely, let $\alpha:\Gamma\acts X$ be a continuous action
of a countably infinite discrete group on a nonempty compact Hausdorff
zero-dimensional space, and
write $\cM_\Gamma(X)$ for its invariant regular Borel probability measures.
For clopen sets $A,B\subseteq X$, write $A\precsimG B$ if $A$
admits a finite clopen partition whose pieces can be moved by
elements of $\Gamma$ to pairwise disjoint subsets of $B$.
The action has \emph{dynamical comparison} if
\begin{equation}\label{eq:intro-comparison}
   \mu(A)<\mu(B)
   \quad\text{for every }\mu\in\cM_\Gamma(X)
   \quad\Longrightarrow\quad
   A\precsimG B
\end{equation}
for all nonempty clopen sets $A,B$.
This is the clopen formulation of the comparison property studied
in \cite{Kerr,Melleray}. The Glasner--Weiss theorem immediately implies dynamical comparison
for Cantor minimal $\mathbb Z$-systems. Indeed, if
$g\in[[T]]$ satisfies $g(A)\subseteq B$, then on a finite clopen
partition of $A$ the map $g$ agrees with powers of $T$, which gives
$A\precsim_T B$.

Beyond $\mathbb Z$-actions, positive results for dynamical comparison
have been obtained under additional assumptions on the acting group.
Downarowicz and Zhang proved it for zero-dimensional compact metrizable
actions of groups whose finitely generated subgroups all have
subexponential growth
\cite[Theorem~6.33 and Proposition~6.18]{DownarowiczZhang}.
Naryshkin proved comparison for minimal actions of finitely generated
groups of polynomial growth on arbitrary compact metrizable spaces
\cite[Theorem~A]{Naryshkin}.

Melleray asked whether every minimal Cantor action of a countable group
has dynamical comparison \cite[Question~6.1]{Melleray}. In this
generality the answer is negative: Boldrini and Prasad constructed
topologically free minimal Cantor actions of $\mathbb F_\infty$
without comparison, including examples admitting invariant probability
measures \cite{BoldriniPrasad}.
Our result establishes comparison for all actions of all
countably infinite discrete amenable groups on nonempty compact Hausdorff
zero-dimensional spaces. There is no growth assumption on the group,
and the action need not be minimal or free.

The comparison problem has a natural algebraic formulation through the
\emph{clopen type semigroup} $T(\alpha)$, which records finite clopen
equidecomposition classes in finite disjoint unions of labelled copies
of $X$. Addition is induced by disjoint union, and $a\leq b$ means that
a representative of $a$ is equidecomposable with a clopen subset of a
representative of $b$. For a clopen set $A\subseteq X$, we write $[A]$
for its type. Each invariant measure $\mu$ evaluates a type by summing
its values over the labelled copies. This evaluation is denoted by
$\mu(a)$. The precise construction is recalled in
Section~\ref{sec:preliminaries}.

A monoid is \emph{cancellative} if $a+c=b+c$ implies $a=b$, and
\emph{almost unperforated} if $(n+1)a\leq nb$ implies $a\leq b$
for every $n\geq1$. Melleray also asked whether $T(\alpha)$ is
cancellative for every minimal amenable Cantor action
\cite[Question~6.2]{Melleray}.  We give an affirmative answer to
this question.

In fact, cancellation is part of a stronger structural picture.
The type semigroup is canonically related to the group of coinvariants.
Let $H(\alpha)$ be the quotient of $C(X,\mathbb Z)$ by the subgroup
generated by $f-f\circ\gamma^{-1}$, with $f\in C(X,\mathbb Z)$ and
$\gamma\in\Gamma$. Write $[f]_H$ for the class of $f$, and
$H(\alpha)^+$ for the image of the nonnegative integer-valued
functions. There is a canonical map
\[
 \pi:T(\alpha)\to H(\alpha)^+.
\]
As in the case of a single homeomorphism, the topological full group
$[[\alpha]]$ consists of homeomorphisms which agree with elements of
$\Gamma$ on a finite clopen partition. A type $c\in T(\alpha)$ is an \emph{order unit} if every type is
bounded above by an integer multiple of $c$; equivalently,
$[X]\leq mc$ for some integer $m\geq1$. Precise definitions and
conventions for these objects are given in
Section~\ref{sec:preliminaries}.

\begin{maintheorem}\label{thm:main-comparison}
Let $\alpha:\Gamma\acts X$ be a continuous action of a countably
infinite discrete amenable group on a nonempty compact Hausdorff
zero-dimensional space.
\begin{enumerate}[label=\textup{(\roman*)}]
\item If $a,b\in T(\alpha)$ satisfy
$
   \mu(a)<\mu(b)
$ for every $\mu\in\cM_\Gamma(X),
$
then
\[
   b=a+c
\]
for some order unit $c\in T(\alpha)$.

Moreover, if clopen sets $A,B\subseteq X$ satisfy
$
   \mu(A)<\mu(B)
$ for every $\mu\in\cM_\Gamma(X),
$
then there exists $g\in[[\alpha]]$ such that
\[
   g^2=\id_X,\qquad
   g(A)\subseteq B,\qquad
   \Gamma\bigl(B\setminus g(A)\bigr)=X.
\]
In particular, $\alpha$ has dynamical comparison.
\end{enumerate}

If, in addition, $X$ is a Cantor space and $\alpha$ is minimal, then
\begin{enumerate}[label=\textup{(\roman*)},start=2]
\item $T(\alpha)$ is cancellative and almost unperforated;
\item the canonical cone $H(\alpha)^+$ is proper, and
      $\pi:T(\alpha)\to H(\alpha)^+$ is an isomorphism of
      ordered monoids;
\item if clopen sets $A,B\subseteq X$ satisfy $[A]=[B]$ in
      $T(\alpha)$, then there exists $g\in[[\alpha]]$ such that
      $g(A)=B$.
\end{enumerate}
\end{maintheorem}

Item~\textup{(iv)} answers
\cite[Question~6.3]{Melleray} for minimal amenable Cantor actions.
The conclusion in \textup{(i)} also extends the Glasner--Weiss
comparison theorem to arbitrary zero-dimensional amenable actions.
Both the full-group realization and the order-unit remainder in
\textup{(i)} follow from dynamical comparison.

Minimality is nevertheless essential for the remaining structural
part of the theorem. In
Subsection~\ref{subsec:nonminimal-counterexample},  we
give a free nonminimal Cantor $\mathbb Z$-action for which the
cancellativity assertion in \textup{(ii)}, both conclusions in
\textup{(iii)}, and the exact full-group conclusion in \textup{(iv)}
all fail. The example does not determine whether the
almost-unperforation assertion in \textup{(ii)} continues to hold
without minimality.

\vspace{.3cm}

The main new ingredient is the following $n=2$ form of almost
unperforation.  It requires neither minimality, freeness, nor
metrizability. 
Together with a finite clopen matching argument, it yields the general
comparison statement in Theorem~\ref{thm:main-comparison}\,\textup{(i)};
in the minimal Cantor case, it also yields full almost unperforation
and cancellation.

\begin{maintheorem}\label{thm:main-three-two}
Let $\alpha:\Gamma\acts X$ be a continuous action of a countably
infinite discrete amenable group on a compact Hausdorff
zero-dimensional space.  Then
\[
   3a\leq2b\quad\Longrightarrow\quad a\leq b
   \qquad(a,b\in T(\alpha)).
\]
\end{maintheorem}
Although Theorem~\ref{thm:main-three-two} treats only the case
$n=2$, in the minimal setting it is enough to recover full almost
unperforation.  Indeed, minimality makes $T(\alpha)$ simple, so every
nonzero type is an order unit.  Writing $u=[X]$, for each nonzero
$c\in T(\alpha)$ one may choose $h$ so that $u\leq 2^h c$.
If $3^h b\leq u$, then
\[
   3^h b\leq u\leq 2^h c
\]
and repeated application of Theorem~\ref{thm:main-three-two} yields
$b\leq c$.  This gives weak comparability, which for minimal actions
implies almost unperforation and, together with the standard
type-semigroup results, dynamical comparison and cancellation
\cite[Propositions~2.18 and~2.20]{Melleray}.

For dynamical comparison without minimality, the main difficulty is
that a nonzero type need not be an order unit. Starting from a strict
inequality under all invariant measures, a uniform F\o lner argument
first produces a clopen injection after a common finite amplification.
Only finitely many group elements occur in this injection, and fixing
them gives uniform control of every target remainder arising during
the construction.
The partial matching is then improved by finitely many local
rearrangements. If some source remains unmatched, sufficiently many
disjoint copies of this remainder can be packed into the original
amplified source, while a fixed number of copies of the corresponding
target remainder already dominates the whole space. Choosing the
parameters appropriately, these two estimates place the two
remainders within the range of repeated applications of
Theorem~\ref{thm:main-three-two}. The remaining source can therefore
be absorbed into the remaining target, completing the comparison. The details appear in
Subsection~\ref{subsec:nonminimal-comparison}.

\vspace{.3cm}

A related question concerns the invariant measures themselves.
Given a minimal amenable Cantor action, can a minimal
homeomorphism of the same space have exactly the same invariant
probability measures? This question is discussed in
\cite[Introduction]{Melleray} and is related to the realization
of prescribed families of measures by minimal homeomorphisms
\cite{IbarluciaMelleray}.

Glasner's affine evaluation maps provide a way to formulate the
relevant condition using clopen sets \cite{Glasner}.
Put $Q=\cM_\Gamma(X)$, and write $\Clop(X)$ for the clopen
subsets of $X$. The \emph{affine evaluation map} assigns to
$A\in\Clop(X)$ the function
\[
   \widehat A:Q\to\mathbb R,
   \qquad \widehat A(\mu)=\mu(A).
\]
It is \emph{good} if $\widehat A<\widehat B$ pointwise on $Q$
implies that some clopen $D\subseteq B$ satisfies
$\widehat D=\widehat A$.
Glasner asked whether this holds for the invariant measures of
every minimal amenable Cantor action
\cite[Question~11.2]{Glasner}.
Comparison gives an affirmative answer: the disjoint translates
of the pieces of $A$ form the required subset $D$ of $B$.

The same evaluation also relates the coinvariant group to the
ordered group associated with $Q$. Following
\cite[Section~4]{Glasner}, let $G_Q$ be the quotient of
$C(X,\mathbb Z)$ in which two functions are identified when their
integrals agree against every measure in $Q$. The subgroup
$\Inf(H(\alpha))$ consists of the coinvariant classes whose
integrals vanish against every measure in $Q$. See Subsection~\ref{subsec:coinvariant-prelim} for the precise
definitions and conventions used here.

\begin{maintheorem}\label{thm:main-aem}
Let $\alpha:\Gamma\acts X$ be a minimal action of a countably
infinite discrete amenable group on a Cantor space, and put
$Q=\cM_\Gamma(X)$.  Then:
\begin{enumerate}[label=\textup{(\roman*)}]
\item the affine evaluation map associated with $Q$ is good\footnote{This statement also holds without
minimality, by Corollary~\ref{cor:nm-direct-consequences}.};
\item there is a minimal homeomorphism $T:X\to X$ such that
      $\cM_T(X)=Q$;
\item the identity on $C(X,\mathbb Z)$ induces an isomorphism of
      unital ordered groups
      \[
         H(\alpha)/\Inf(H(\alpha))\cong G_Q,
      \]
      where $q$ denotes  the quotient map. The quotient has positive cone $q(H(\alpha)^+)$
      and order unit $\bar u:=q([1_X]_H)$.  This quotient is a countable
      simple dimension group, its normalized state space is
      canonically affinely homeomorphic to $Q$, and
the scale of $G_Q$, with $u_Q:=[1_X]_Q$, has the form
$$
[0, u_Q]_{G_Q} =  \{[1_A]_Q :  A \in \Clop(X)\}.
$$
Under the canonical isomorphism, this becomes
$$
[0,\bar u]_{H(\alpha)/\Inf(H(\alpha))}  =\{q([\one_A]_H):A\in\Clop(X)\}.
$$
\end{enumerate}
\end{maintheorem}
Once goodness is established, the realization in
part~\textup{(ii)} follows from \cite[Proposition~11.1]{Glasner},
and the dimension-group, state-space, and clopen-scale
conclusions follow from
\cite[Theorem~4.3 and Proposition~5.10]{Glasner}.

The quotient in Theorem~\ref{thm:main-aem} \textup{(iii)} also
identifies exactly what invariant measures fail to detect.
The passage 
from $H(\alpha)$ to $G_Q$ forgets precisely the
infinitesimal subgroup.  In
Subsection~\ref{subsec:exact-equidecomposition}, we give this kernel
a concrete dynamical interpretation: infinitesimals measure exactly
the gap between equality of affine evaluations and actual
clopen equidecomposition by the original action.

Write $A\sim_\Gamma B$ for clopen equidecomposability and set
\[
   \mathcal H_Q
  : =
   \{h\in\Homeo(X):h_*\mu=\mu\text{ for every }\mu\in Q\}.
\]
Let $B_\alpha$ denote the subgroup of coboundaries defining
$H(\alpha)$, and put
\[
   N_Q:=\left\{f\in C(X,\mathbb Z):\int f\,d\mu=0
                    \text{ for every }\mu\in Q\right\}.
\]
The next theorem identifies the infinitesimal obstruction and
classifies the full-group orbits within each nontrivial fibre of
affine evaluation.

\begin{maintheorem}\label{thm:main-infinitesimals}
Let $\alpha:\Gamma\acts X$ be a minimal action of a countably
infinite discrete amenable group on a Cantor space, and put
$Q=\cM_\Gamma(X)$.
\begin{enumerate}[label=\textup{(\roman*)}]
\item The following conditions are equivalent:
  \begin{enumerate}[label=\textup{(\alph*)}]
  \item $\Inf(\Halpha)=0$;
  \item the evaluation map
        \[
           \rho_Q:\Talpha\to\Aff(Q),
           \qquad \rho_Q(a)(\mu)=\mu(a),
        \]
        is injective;
  \item for all clopen $A,B\subseteq X$, equality
        $\widehat A=\widehat B$ implies $A\sim_\Gamma B$;
        equivalently, $g(A)=B$ for some $g\in[[\alpha]]$;
  \item $\overline{[[\alpha]]}=\mathcal H_Q$, where the closure
        is taken in $\Homeo(X)$ with respect to the topology of
        uniform convergence of maps and inverses;
  \item $B_\alpha=N_Q$.
  \end{enumerate}
\item For every nonempty proper clopen set $A\subset X$, put
  \[
     \mathscr F_A
       :=\{B\in\Clop(X):\widehat B=\widehat A\}.
  \]
  Then the map
  \[
     \mathscr F_A/[[\alpha]]\longrightarrow\Inf(\Halpha),
     \qquad [[\alpha]]\cdot B\longmapsto[\one_B-\one_A]_H,
  \]
  is a bijection.
\end{enumerate}
\end{maintheorem}
  Thus every infinitesimal occurs as the difference
between $A$ and a clopen set with the same affine evaluation,
and these differences distinguish the $[[\alpha]]$-orbits in
$\mathscr F_A$.  The quotient here is a set of orbits, not a
quotient group.

\vspace{.3cm}

For free actions, comparison also has consequences for almost
finiteness and crossed products. For minimal homeomorphisms of
infinite compact metrizable spaces of finite dimension, Toms and Winter proved
$\mathcal Z$-stability of the associated crossed products
\cite{TomsWinter}.   Kerr's notion of almost finiteness
provides a dynamical counterpart of finite-dimensional approximation
in $C^*$-algebras \cite{Kerr}.  Kerr and Szab\'o proved that a free
action of a countably infinite discrete amenable group on a compact
metrizable space is almost finite if and only if it has comparison
and the small boundary property
\cite[Theorem~6.1]{KerrSzabo}.
Kerr and Naryshkin established almost finiteness for all free actions
of countably infinite elementary amenable groups on finite-dimensional
compact metrizable spaces
\cite[Theorem~A]{KerrNaryshkin}, and Naryshkin subsequently showed that
almost finiteness passes from the restriction to an infinite normal
subgroup to the full free action
\cite[Theorem~A]{NaryshkinExtensions}.

The almost-finiteness conjecture stated explicitly by Naryshkin in
\cite{NaryshkinExtensions} asks whether every free action of a countably
infinite discrete amenable group on a finite-dimensional compact
metrizable space is almost finite.
By Theorem~\ref{thm:main-comparison} \textup{(i)}, every action of a
countably infinite discrete amenable group on a compact
zero-dimensional space has dynamical comparison. Hence, in the free
metrizable case, zero-dimensionality together with
\cite[Theorem~6.1]{KerrSzabo} gives almost finiteness.
The reduction of Kerr and Szab\'o then passes from all free
zero-dimensional actions to all free actions with the topological
small boundary property \cite[Theorem~7.6]{KerrSzabo}.
\cite[Corollary~7.7]{KerrSzabo} then gives the finite-dimensional case.
In particular, this settles Naryshkin's conjecture.

\begin{maintheorem}\label{thm:main-almost-finite}
Let $\Gamma$ be  a countably infinite discrete amenable group, and let
$\Gamma\acts X$ be a free action  on a nonempty compact metrizable
space  with the topological small boundary property. Then the action
 is almost finite. 
 
 In particular,   every free action of $\Gamma$ on a
nonempty finite-dimensional compact metrizable space is almost finite.
\end{maintheorem}

For free minimal actions, Theorem~\ref{thm:main-almost-finite} also
yields regularity of the associated crossed products.
Recall that the Toms--Winter conjecture predicts the equivalence of
several regularity properties for separable, simple, unital,
infinite-dimensional nuclear $C^*$-algebras, including finite nuclear
dimension, $\mathcal Z$-stability, and strict comparison of positive
elements; see \cite[Conjecture~9.3]{WinterZachariasNuclearDimension}.
Kerr and Szab\'o established this equivalence for crossed products
arising from free minimal actions of countably infinite amenable groups
with the small boundary property \cite[Corollary~9.5]{KerrSzabo}.
In the discussion following that corollary, they further expressed
the expectation that all such crossed products satisfy these
regularity conditions, noting that this was already known for
$\mathbb Z$-actions by the work of Elliott and Niu \cite{ElliottNiu}.
The following theorem verifies this expectation under the stronger
hypothesis of the topological small boundary property.

\begin{maintheorem}\label{thm:main-crossed-products}
Let $\Gamma$ be a countably infinite discrete amenable group, and let
$\Gamma\acts X$ be a free minimal action on a nonempty compact
metrizable space with the topological small boundary property.
Then the reduced crossed product
$
    A=C(X)\rtimes_r\Gamma
$
satisfies 
\begin{enumerate}[label=\textup{(\roman*)}]
\item $A$ is $\mathcal Z$-stable;
\item $\dim_{\mathrm{nuc}}(A)\leq 1$;
\item the Cuntz semigroup $W(A)$ is almost unperforated;
\item $A$ has strict comparison of positive elements.
\end{enumerate}
In particular, these conclusions hold for every free minimal action
of $\Gamma$ on a nonempty finite-dimensional compact metrizable space.
\end{maintheorem}

Here $\mathcal Z$ denotes the Jiang--Su algebra, and
$\mathcal Z$-stability means
\[
A\otimes_{\min}\mathcal Z\cong A.
\]
Since the other $C^*$-algebraic regularity notions appearing above are
used only in this consequence and its proof, we do not recall their
definitions here. For nuclear dimension, see
\cite[Definition~2.1]{WinterZachariasNuclearDimension}.
We write $W(A)$ for the Cuntz semigroup of positive elements in
matrix algebras over $A$; for this semigroup and strict comparison
of positive elements, see \cite[Section~4]{RordamZStable}.
In the finite-dimensional case, the $\mathcal Z$-stability conclusion
extends the result of Toms and Winter \cite{TomsWinter} from minimal
$\mathbb Z$-actions to free minimal actions of arbitrary countably
infinite discrete amenable groups.

\medskip
\medskip

\noindent\textbf{Organization.} Section~\ref{sec:preliminaries} collects the notation and external
results.  Section~\ref{sec:matching} develops the clopen matching
arguments, and Section~\ref{sec:bisection-matching} applies them to
finite-order bisections.  Section~\ref{sec:three-two} proves
Theorem~\ref{thm:main-three-two} and
Theorem~\ref{thm:main-comparison} \textup{(i)--(ii)}.
Section~\ref{sec:consequences} completes the proof of
Theorem~\ref{thm:main-comparison}, proves
Theorem~\ref{thm:main-aem}, and establishes
Theorem~\ref{thm:main-infinitesimals} on exact equidecomposition and
infinitesimals. 
Section~\ref{sec:almost-finite} proves
Theorem~\ref{thm:main-almost-finite} using the reduction of Kerr and
Szab\'o, thereby settling Naryshkin's almost-finiteness conjecture,
and derives the crossed-product regularity conclusions of
Theorem~\ref{thm:main-crossed-products}.

\section{Preliminaries}\label{sec:preliminaries}
Throughout the paper, we write
$\mathbb N=\{1,2,\ldots\},$ and $ [n]=\{1,\ldots,n\}$ for $n\in\mathbb N$.
If $S$ is a subset of a group $G$, then $\langle S\rangle$
denotes the subgroup of $G$ generated by $S$.

\subsection{Spaces, actions, and invariant measures}\label{subsec:spaces}
All topological spaces are Hausdorff, and all group actions are continuous left
actions of discrete groups.  For a topological space $Y$, we write $\Homeo(Y)$ for the group
of homeomorphisms of $Y$.  
 A space is \emph{zero-dimensional} if it has a basis of
clopen sets, and a \emph{Cantor space} is a nonempty compact metrizable
zero-dimensional space with no isolated points.  We write
 $\Clop(X)$ for the clopen subsets of $X$.

An action $\Gamma\acts X$ is \emph{minimal} if every orbit is dense,
and \emph{free} if $\gamma x=x$ implies $\gamma=1_\Gamma$.
For a compact space $X$, let $\cM(X)$ be its regular Borel probability
measures, with the weak* topology, and let $\cM_\Gamma(X)$ be the
 subset of invariant measures. 
 
A countably infinite discrete group $\Gamma$ is \emph{amenable}
if there is a sequence $(F_n)_{n\in\mathbb N}$ of
nonempty finite subsets of $\Gamma$ such that
\[
  \lim_{n\to\infty} \frac{|\gamma F_n\mathbin\triangle F_n|}{|F_n|}=0
   \qquad(\gamma\in\Gamma).
\]
Such a sequence is called a \emph{F\o lner sequence}.
If  $X$ is nonempty and compact and $\Gamma$ is amenable, then
$\cM_\Gamma(X)\neq\varnothing$.
When $X$ is compact metrizable, $\cM_\Gamma(X)$ is compact and
metrizable.  If it is nonempty, the ergodic decomposition theorem
implies that it is a Choquet simplex.

For a minimal action on a compact space, every invariant probability
measure has full support, since its support is a nonempty closed
invariant set.  For a minimal Cantor action, every orbit is infinite,
so every invariant probability measure is atomless. 

A \emph{clopen partial homeomorphism} from $L$ to $R$ is a
homeomorphism $\phi:D\to D'$ between clopen subsets $D\subseteq L$
and $D'\subseteq R$. Inverses, compositions on their natural domains, and restrictions to clopen subsets are again of this form.

For an action of a finite group $J$ on $Z$, a
\emph{fundamental domain} meets each orbit exactly once.
The Cantor-space statement is
\cite[Lemma~A.2]{GiordanoMatuiPutnamSkau2010}.  By the same argument, we have the following for compact
zero-dimensional spaces.
\begin{lemma}\label{lem:finite-group-domain}
A free action of a finite group $J$ on a compact zero-dimensional
space $Z$ has a clopen fundamental domain $D$.  In particular,
$Z=\bigsqcup_{j\in J}jD$.
\end{lemma}

\subsection{The clopen type semigroup and comparison}\label{subsec:types}
The notion of a type semigroup goes back to Tarski's work on equidecomposition and amenability. 
Its clopen version for actions on zero-dimensional compact spaces was introduced by 
R{\o}rdam and Sierakowski, \cite{RordamSierakowski}, and was subsequently used by Kerr \cite{Kerr} and Ma  
\cite{Ma} in the study of dynamical comparison and almost unperforation.
Melleray developed a systematic treatment of these clopen type semigroups \cite{Melleray}.
We now recall the definition of this semigroup associated with an action of a group $\Gamma$.

Let $\alpha:\Gamma\acts X$ be a continuous action on a compact
zero-dimensional space.
Put $\widetilde X:=X\times\mathbb N$, with $\mathbb N$ discrete,
and write $\Clop_b(\widetilde X)$ for the bounded clopen subsets
of $\widetilde X$, namely those contained in
$X\times[m]$ for some $m$.
Two such sets $D,D'$ are \emph{equidecomposable} (denoted $D \sim_\alpha D'$) if $D$ has a finite
clopen partition into pieces $A_t\times\{i_t\}$ and there are
$\gamma_t\in\Gamma$ and sheet indices $j_t$ such that
\[
   D'=\bigsqcup_t\gamma_tA_t\times\{j_t\}.
\]
Thus each piece is mapped by $(x,i_t)\mapsto(\gamma_tx,j_t)$.
The equivalence class of $D$ is denoted $[D]$, and 
$T(\alpha)$, which is called the \emph{clopen type
semigroup}, is the set of all such classes. 

For $D,E\in\Clop_b(\widetilde X)$, define $[D]+[E]$ by
taking their union after relabelling their sheet coordinates to
be disjoint.  This makes $T(\alpha)$ a commutative monoid with
zero $[\varnothing]$.  Its algebraic preorder is
\[
   a\leq b\quad\Longleftrightarrow\quad
   a+c=b\text{ for some }c\in\Talpha.
\]
Equivalently, $a\leq b$ means that a representative of $a$ is
equidecomposable with a clopen subset of a representative of $b$.
Note that we do not assume that this preorder is antisymmetric.   For simplicity, we  write $[A]=[A\times\{1\}]$ for $A\in\Clop(X)$ and $u=[X]$.

 For $f\in C(X,\mathbb Z_{\geq0})$, put
\[
   \Sigma(f):=\bigsqcup_{j\geq1}\{x:f(x)\geq j\}\times\{j\},
   \qquad [f]_T:=[\Sigma(f)].
\]
Since $f$ is bounded, $\Sigma(f)\in\Clop_b(\widetilde X)$.
For $D\in\Clop_b(\widetilde X)$, define $$m_D(x):=|\{j:(x,j)\in D\}|.$$ Then
$$[D]=[m_D]_T.$$  
  Addition of functions
corresponds to addition of types, and $f\leq g$ pointwise implies
$[f]_T\leq[g]_T$.  We keep the subscript when it is necessary to
distinguish a type from a coinvariant class.

An invariant probability measure evaluates types by
\[
   \mu([D])=\sum_j\mu(D_j),\qquad\text{where }
   D_j:=\{x\in X:(x,j)\in D\}.
\]
This is independent of the representative and gives an additive map
with $\mu(u)=1$.  Such maps are called normalized \emph{states}.
Since
$
   f=\sum_{j\ge1}\one_{\{f\ge j\}},
$
with only finitely many nonzero terms, we have \begin{equation}\label{eq:formu}
    \mu([f]_T)=\int f\,d\mu.
\end{equation}
For a minimal action these states are \emph{faithful}, meaning that
$\mu(a)=0$ implies $a=0$.

For clopen $A,B\subseteq X$, write $A\precsimG B$ when
$[A]\leq[B]$.  Equivalently, there is a finite clopen partition
$A=\bigsqcup_t A_t$ and elements $\gamma_t\in\Gamma$ whose images
$\gamma_t A_t$ are pairwise disjoint subsets of $B$.
The action has \emph{dynamical comparison} if
\[
   \mu(A)<\mu(B)\quad(\mu\in\cM_\Gamma(X))
   \quad\Longrightarrow\quad A\precsimG B
\]
for all nonempty clopen $A,B$.
For clopen $A,B\subseteq X$, we also write $A\sim_\Gamma B$ when
$[A]=[B]$ in $\Talpha$.  Thus $A\precsimG B$ means that
$A\sim_\Gamma C$ for some clopen $C\subseteq B$.

A commutative monoid $M$ is \emph{conical} if $a+b=0$ forces
$a=b=0$, and has \emph{refinement} if every equality
$a_1+a_2=b_1+b_2$ can be written as
$a_i=c_{i1}+c_{i2}$ and $b_j=c_{1j}+c_{2j}$.
An element $v$ is an \emph{order unit} if every element is at most
an integer multiple of $v$, and $M$ is \emph{simple} if every nonzero
element is an order unit.  It is \emph{cancellative} if
$a+c=b+c$ implies $a=b$.  It is \emph{almost unperforated} if for all $n\ge 1$ and $a,b \in M$
\[
   (n+1)a\leq nb\quad\Longrightarrow\quad a\leq b.
\]
For a simple monoid with order unit $u$, \emph{weak comparability (with respect to $u$)}
means that for every $a\neq0$ there is $k\geq1$ such that  for all $b$,
$kb\leq u$ implies $b\leq a$.

The type monoid $T(\alpha)$ is conical and has refinement, and
minimality implies simplicity; see
\cite[Definition~2.3 and the following paragraph]{Melleray}.
In particular, $u=[X]$ is an order unit.

When $\alpha$ is minimal, we specialize the preceding notion of weak
comparability to this distinguished order unit $u=[X]$.  Thus, for a minimal action,
\emph{weak comparability of $T(\alpha)$} will always mean weak
comparability with respect to the order unit
$u=[X]$.

We record 
the comparison results used later.
\begin{theorem}\label{thm:known-type-results}
Let $\alpha:\Gamma\acts X$ be an action of a countably infinite
discrete group on a compact zero-dimensional space.
\begin{enumerate}[label=\textup{(\roman*)}]
\item If $\alpha$ is minimal, then weak comparability of $\Talpha$
      is equivalent to almost unperforation
      \cite[Proposition~2.20(1)]{Melleray}.
\item If $\alpha$ is minimal and $\cM_\Gamma(X)\neq\varnothing$,
      either condition in \textup{(i)} implies that $\Talpha$ is
      cancellative \cite[Proposition~2.20(2)]{Melleray}.
\item If $\alpha$ is minimal, then dynamical comparison is equivalent
      to almost unperforation \cite[Proposition~2.18]{Melleray}.
\item For any $\alpha$, dynamical comparison is equivalent to the
      following assertion: if $a,b\in\Talpha$ are nonzero and
      $\mu(a)<\mu(b)$ for every $\mu\in\cM_\Gamma(X)$, then
      $a\leq b$ \cite[Proposition~2.15]{Melleray}.
\end{enumerate}
\end{theorem}

\subsection{Coinvariants, full groups, and affine evaluation}
\label{subsec:coinvariant-prelim}
Use the left action on functions
$(\gamma\cdot f)(x)=f(\gamma^{-1}x)$ and set
\[
   B_\alpha:=\left\langle f-\gamma\cdot f:
      f\in C(X,\mathbb Z),\ \gamma\in\Gamma\right\rangle,
   \qquad \Halpha:=C(X,\mathbb Z)/B_\alpha.
\]
For $f\in C(X,\mathbb Z)$ its class is $[f]_H\in H(\alpha)$.
We call $H(\alpha)$ the \emph{group of coinvariants} of the action.

For an abelian group $D$, a \emph{cone} is an additive submonoid
$D^+\subseteq D$ containing $0$.  It is \emph{proper} if
\[
   D^+\cap(-D^+)=\{0\}.
\]
A proper cone defines a partial order on $D$ by
\[
   x\leq y
   \quad\Longleftrightarrow\quad
   y-x\in D^+.
\]
For $x\leq y$ in an ordered abelian group $D$, write
\[
   [x,y]_D:=\{d\in D:x\leq d\leq y\}
\]
for the corresponding order interval.
The canonical additive cone is
\[
   \Halpha^+:=\{[f]_H:f\in C(X,\mathbb Z_{\geq0})\}.
\]
Every equidecomposition preserves coinvariant classes, so
\begin{equation}\label{eq:iso}
       \pi:\Talpha\to\Halpha^+,
   \qquad \pi([f]_T)=[f]_H
\end{equation}
is an additive surjection.  The following criterion is
\cite[Proposition~2.5]{Melleray}.
\begin{proposition}\label{prop:known-coinvariant}
The map $\pi$ is injective if and only if $\Talpha$ is cancellative.
\end{proposition}

The \emph{topological full group} $[[\alpha]]$ consists of those
homeomorphisms of $X$ which agree with elements of $\Gamma$ on the
parts of some finite clopen partition.
For a nonempty compact convex $Q\subseteq\cM(X)$, let $\Aff(Q)$
denote the real vector space of continuous affine functions on $Q$.
The \emph{affine evaluation map} sends $A\in\Clop(X)$ to the continuous
affine function $\widehat A(\mu)=\mu(A)$ on $Q$.
It is \emph{good} if $\widehat A<\widehat B$ pointwise on $Q$
implies that some clopen $D\subseteq B$ satisfies
$\widehat D=\widehat A$.
Write $A\sim_Q B$ when $\widehat A=\widehat B$, and put
\[
   \mathcal H_Q:=\{h\in\Homeo(X):h_*\mu=\mu\text{ for every }\mu\in Q\}.
\]
  When
$Q=\cM_\Gamma(X)$, one always has
$[[\alpha]]\subseteq\mathcal H_Q$ and
$A\sim_\Gamma B\Rightarrow A\sim_Q B$.

We recall some notation from  \cite[Section~4]{Glasner}.
Define
\[
   N_Q:=\left\{f\in C(X,\mathbb Z):\int f\,d\mu=0
                         \text{ for every }\mu\in Q\right\},
\]
\[
   G_Q:=C(X,\mathbb Z)/N_Q,
\]
with the cone
\[
   G_Q^+:=\{0\}\cup
   \left\{[f]_Q:\int f\,d\mu>0\text{ for every }\mu\in Q\right\}
\]
and order unit $[\one_X]_Q$.  
When $Q=\cM_\Gamma(X)$, the evaluation
$\widehat{[f]_H}(\mu)=\int f\,d\mu$ is well defined, and
\[
   \Inf(\Halpha):=\{h\in\Halpha:\widehat h(\mu)=0
                                    \text{ for all }\mu\in Q\}
\]
is its infinitesimal subgroup (\cite[Section 6]{Glasner}).

The following result is
\cite[Proposition~11.1]{Glasner}.
\begin{proposition}\label{prop:known-realization}
For a minimal Cantor action of a countably infinite discrete amenable
group, the affine evaluation map of $\cM_\Gamma(X)$ is good if and only if there is
a minimal homeomorphism $T:X\to X$ such that $\cM_T(X)=\cM_\Gamma(X)$.
\end{proposition}

An ordered abelian group $(D,D^+)$ is a \emph{dimension group} if
$D=D^+-D^+$, it is unperforated, and it has Riesz interpolation.
Here \emph{unperforation} means $nd\geq0\Rightarrow d\geq0$ for $n\geq1$ and
\emph{interpolation} means that $a_i\leq b_j$ for $i,j\in\{1,2\}$ implies
$a_i\leq c\leq b_j$ for some $c\in D$ and all $i,j$.
When $D=D^+-D^+$, an element $v\in D^+$ is an order unit of the
monoid $D^+$ if and only if, for every $d\in D$, there exists
$n\geq1$ such that
\[
   -nv\leq d\leq nv.
\]  Write
\[
   S(D,v):=\{\tau:D\to\mathbb R:\tau\text{ is additive and positive},
                                      \ \tau(v)=1\}
\]
for its normalized state space, with the topology of pointwise
convergence.

We record the additional AEM results used below.  These are
\cite[Theorem~4.3, Proposition~5.10, Corollary~5.11,
and Proposition~3.3]{Glasner}, respectively.
\begin{proposition}\label{prop:known-aem-structure}
Let $X$ be a Cantor space and let $Q$ be a nonempty compact Choquet
simplex of atomless full-support probability measures.  Suppose that
the associated geometric AEM is good, and put $u_Q:=[\one_X]_Q$.
\begin{enumerate}[label=\textup{(\roman*)}]
\item 
$(G_Q,G_Q^+,u_Q)$ is a countable simple dimension group, and
      $\mu\mapsto([f]_Q\mapsto\int f\,d\mu)$ identifies $Q$ affinely
      and homeomorphically with $S(G_Q,u_Q)$.
\item For every clopen $B\subseteq X$,
      \[
        [0,[\one_B]_Q]_{G_Q}
        =\{[\one_C]_Q:C\subseteq B\text{ is clopen}\}.
      \]
\item If
      $
        J_Q:=\left\langle\one_A-\one_B:
                    A,B\in\Clop(X),\ A\sim_Q B\right\rangle,
      $
      then $J_Q=N_Q$.
\item If $A\sim_Q B$, some $h\in\mathcal H_Q$ satisfies $h(A)=B$.
\end{enumerate}
\end{proposition}

\subsection{Transformation groupoids and bisections}
\label{subsec:transformation-groupoids}

We use the standard groupoid language of
\cite[Chapter~I, Section~1]{Renault1980}.
A \emph{groupoid} $\cG$ consists of invertible arrows between 
\emph{units} in
$\cG^{(0)}$.  An arrow $\gamma$ goes from $s(\gamma)$ to $r(\gamma)$,
and $\gamma\eta$ is defined when $s(\gamma)=r(\eta)$ and means first
$\eta$, then $\gamma$.  Multiplication is associative on composable
triples, each unit $z$ has an identity arrow $1_z$, and
$\gamma^{-1}\gamma=1_{s(\gamma)}$,
$\gamma\gamma^{-1}=1_{r(\gamma)}$.
In a topological groupoid these operations and the source and range
maps are continuous.  A group $K$ is viewed as a groupoid with a single unit $*$, with
arrow space $K$, source and range maps
\[
s(k)=r(k)=*, \qquad k\in K,
\]
and with groupoid multiplication and inversion given by the group
multiplication and inversion in $K$.  Thus every pair of arrows is
composable.

For an action $H\acts Z$ 
of a countable discrete group $H$
on a compact space $Z$, define
$H\ltimes Z=H\times Z$, with the product topology, by
\[
   (h,z):z\to hz,\qquad
   s(h,z)=z,\quad r(h,z)=hz,
\]
and
\[
   (g,hz)(h,z)=(gh,z),\qquad
   (h,z)^{-1}=(h^{-1},hz).
\]
The unit space is identified with $Z$ via $z\mapsto(1_H,z)$.

A groupoid is \emph{\'etale} if its source map is a local
homeomorphism, and inversion then gives the same property for the range
map. See, for example, \cite{Renault1980}. Here $s$ restricts to a homeomorphism on each open slice
$\{h\}\times Z$, so $H\ltimes Z$ is \'etale.

For $Y\subseteq Z$, its \emph{reduction} is
\[
   (H\ltimes Z)|_Y=\{(h,z):z\in Y,\ hz\in Y\}.
\]
No $H$-invariance of $Y$ is required.  When $Y$ is clopen the
reduction is again a locally compact \'etale groupoid.

A \emph{bisection} is a set of arrows on which both $s$ and $r$ are
injective.  An open bisection $B$ in an \'etale groupoid induces a
partial homeomorphism
\[
   \theta_B=r\circ(s|_B)^{-1}:s(B)\to r(B).
\]
An arbitrary nonopen bisection need not have this topological property.
A bisection is \emph{full over $Y$} if $s(B)=r(B)=Y$.
For bisections $B$ and $D$, multiplication and inverse mean
\[
   BD=\{\gamma\eta:\gamma\in B,\ \eta\in D,
                      \ s(\gamma)=r(\eta)\},\qquad
   B^{-1}=\{\gamma^{-1}:\gamma\in B\}.
\]
Then $BD$ induces the partial homeomorphism
$
\theta_{BD}=\theta_B\circ\theta_D
$
on the domain
$
\{x\in s(D):\theta_D(x)\in s(B)\}.
$  

Let $Y\subseteq X$ be clopen. Each compact-open full bisection $B$ induces a homeomorphism
$\theta_B:Y\to Y$.
The compact-open full bisections of $(H\ltimes X)|_Y$ form a group $\operatorname{Bis}_c((H\ltimes X)|_Y)$
under bisection multiplication. Its identity is
$
   1_Y=\{(1_H,y):y\in Y\},
$
the bisection consisting of all unit arrows over $Y$.

For compact $Z$, a compact-open bisection $B$ of $H\ltimes Z$ meets only
finitely many slices, so it has the form
\[
   B=\bigsqcup_{i=1}^m(\{h_i\}\times D_i),
\]
where the $D_i$ are clopen and pairwise disjoint and the sets $h_iD_i$
are pairwise disjoint.  Conversely, such data define a compact-open
bisection.  The induced map is $\theta_B(z)=h_i z$ on $D_i$.
Note that if $h\neq1_H$ fixes $z$, the arrows $(h,z)$ and
$(1_H,z)$ are different although their endpoints agree.
Consequently, $B^n=1_Y$ requires the product of the labels along
$n$ successive steps to be $1_H$.

\subsection{Topological amenability and cocycles}\label{subsec:amenable}

For a countably infinite discrete group $H$, let
\[
   \ell^1(H):=\left\{m:H\to\mathbb R:
                \|m\|_1:=\sum_{g\in H}|m(g)|<\infty\right\},
\]
and
\[
   \Prob(H):=\{m\in\ell^1(H):m\geq0,\ \|m\|_1=1\}.
\]

Left translation is $(hm)(g):=m(h^{-1}g)$.
A compact action $H\acts Z$ is \emph{topologically amenable} if for
every finite $A\subseteq H$ and $\varepsilon>0$ there is a
norm-continuous map $z\mapsto m_z\in\Prob(H)$ with
\[
   \sup_{z\in Z}\|m_{hz}-hm_z\|_1<\varepsilon
   \qquad(h\in A).
\]
Topological amenability here is a property of the action, so the group
$H$ itself need not be amenable. For further properties of topologically amenable actions, see
\cite{AnantharamanDelaroche2002}.

A continuous map $c:H\times Z\to K$ is a \emph{cocycle} if
\[
   c(gh,z)=c(g,hz)c(h,z).
\]
Equivalently, regarding $K$ as a one-unit groupoid, $c$ induces
a continuous groupoid homomorphism
\[
   \widetilde c:H\ltimes Z\to K,
   \qquad
   \widetilde c(h,z)=c(h,z).
\]
We use the standard equivalence between topological amenability of
$H\curvearrowright Z$ and amenability of its transformation groupoid
$H\ltimes Z$; see \cite{AnantharamanDelarocheRenault}.

\begin{lemma}\label{lem:cocycle-criterion}
Let a countable discrete group $H$ act on a compact metrizable space
$Z$, and let $K$ be a countable discrete amenable group.
If a continuous cocycle $c:H\times Z\to K$ satisfies
\begin{equation}\label{eq:hyper}
     c(h,z)=1_K\quad\Longrightarrow\quad h=1_H,
\end{equation}
  then $H\curvearrowright Z$ is topologically amenable.
\end{lemma}

\begin{proof}
Consider the transformation groupoid
$\mathcal G=H\ltimes Z$ and the homomorphism
$\widetilde c:\mathcal G\to K$ induced by $c$.
Since $H$ is countable discrete and $Z$ is compact metrizable,
$\mathcal G$ is a second-countable locally compact Hausdorff
\'etale groupoid, and hence has the canonical counting Haar system.
By \eqref{eq:hyper},
\[
   \ker\widetilde c
   =
   \{(1_H,z):z\in Z\}
   =
   \mathcal G^{(0)}.
\]
Thus the kernel is the unit groupoid and is amenable.
Since $K$ is amenable, Renault--Williams
\cite[Corollary~4.5]{RenaultWilliams} implies that
$\mathcal G$ is amenable.  The transformation-groupoid equivalence
then gives topological amenability of $H\curvearrowright Z$.
\end{proof}

\subsection{Graphs and matchings}\label{subsec:graph-prelim}

A bipartite multigraph consists of two disjoint vertex sets $L$ and
$R$, an edge set $E$, and endpoint maps
\[
   s:E\to L,
   \qquad
   r:E\to R.
\]
Distinct edges are allowed to have the same pair of endpoints.

Suppose now that $L$ and $R$ are compact zero-dimensional spaces.
We say that the multigraph $(L,R,E)$ admits a
\emph{finite clopen presentation} if there are a finite index set
$I$, nonempty clopen sets
$
   D_i\subseteq L
$, $i\in I$
and homeomorphisms
$
   \phi_i:D_i\to U_i
$
onto clopen sets $U_i\subseteq R$, together with a bijection
$
   \Psi:
   \bigsqcup_{i\in I}(\{i\}\times D_i)
   \longrightarrow E
$
such that, for every $i\in I$ and $x\in D_i$,
\[
   s(\Psi(i,x))=x,
   \qquad
   r(\Psi(i,x))=\phi_i(x).
\]
We equip $E$ with the topology transported by $\Psi$. Equivalently,
$\Psi$ is a homeomorphism from the displayed disjoint union onto
the edge space $E$.  We henceforth identify these two spaces through
$\Psi$.

A \emph{matching} is an edge subset $E_M\subseteq E$ such that
$s|_{E_M}$ and $r|_{E_M}$ are injective.  It is a \emph{clopen
matching} if $E_M$ is clopen.
Set $$A_i:=\{x\in D_i:(i,x)\in E_M\}.$$  For a clopen matching the
$A_i$ are pairwise disjoint clopen sets, as are the $\phi_i(A_i)$,
and
\[
   E_M=\bigsqcup_i\{i\}\times A_i,\qquad
   M:\bigsqcup_i A_i\to\bigsqcup_i\phi_i(A_i),
   \quad M|_{A_i}=\phi_i|_{A_i}.
\]
Thus $M=r\circ(s|_{E_M})^{-1}$ is a clopen partial homeomorphism.
 The matched left and right vertex sets are
$\dom M=s(E_M)$ and $\ran M=r(E_M)$.  The matching is
\emph{left-perfect} if $\dom M=L$, \emph{right-perfect} if
$\ran M=R$, and \emph{perfect} if both hold.

A \emph{walk} is a finite sequence of incident vertices and specified
edges.  Its length is the number of edges.  A \emph{path} is a walk
with no repeated vertex.

An \emph{augmenting path} for $E_M$ is a path of odd length
$2k+1$ with vertices
\[
   x_0,y_0,x_1,y_1,\ldots,x_k,y_k,
   \qquad x_i\in L,\ y_i\in R,
\]
such that
\[
   x_0\notin\dom M,
   \qquad
   y_k\notin\ran M,
\]
and whose specified edges alternate between edges outside and inside
the matching.  More precisely, there are edges
$
   e_i\notin E_M
$, $0\leq i\leq k$
with
$
   s(e_i)=x_i,
$ $
   r(e_i)=y_i,
$
and edges
$
   m_i\in E_M
 $, $1\leq i\leq k
$
with
$
   s(m_i)=x_i,
$ $   r(m_i)=y_{i-1}.
$
Thus an augmenting path begins and ends at unmatched vertices and
alternates between nonmatching edges from $L$ to $R$ and matching
edges from $R$ back to $L$, as illustrated in
Figure~\ref{fig:augmenting-path}.
\begin{figure}[ht]
\centering
\begin{tikzpicture}[
   >=Stealth,
   every node/.style={font=\small}
]
   \node (x0) at (0,0) {$x_0$};
   \node (y0) at (4,0) {$y_0$};

   \node (x1) at (0,-1.4) {$x_1$};
   \node (y1) at (4,-1.4) {$y_1$};

   \node at (0,-2.6) {$\vdots$};
   \node at (4,-2.6) {$\vdots$};

   \node (xk) at (0,-3.9) {$x_k$};
   \node (yk) at (4,-3.9) {$y_k$};

   \node at (0,0.65) {$L$};
   \node at (4,0.65) {$R$};

   \draw[->] (x0) -- node[above] {$e_0\notin E_M$} (y0);
   \draw[->] (y0) -- node[above,sloped] {$m_1\in E_M$} (x1);
   \draw[->] (x1) -- node[above] {$e_1\notin E_M$} (y1);

   \draw[->] (xk) -- node[above] {$e_k\notin E_M$} (yk);
\end{tikzpicture}
\caption{An augmenting path.} {\small\emph{Note.} Arrows indicate the direction of traversal,
rather than the source-to-range orientation of the underlying edges.}
\label{fig:augmenting-path}
\end{figure}
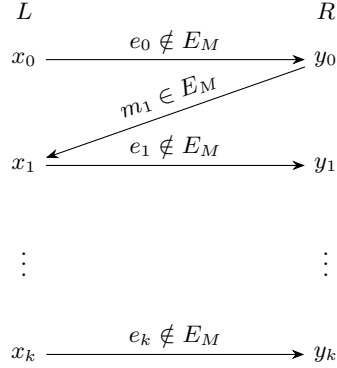

\emph{Augmenting} or \emph{flipping} along it replaces
$\{m_1,\ldots,m_k\}$ by $\{e_0,\ldots,e_k\}$ in $E_M$.
Internal vertices remain matched, although their partners change,
while both endpoints become matched.  Vertex-disjoint paths can be flipped simultaneously.

We also use the following conventions for multigraphs.  A \emph{cycle} may
have length one, in which case it is a \emph{loop}, or length two, in
which case it consists of two distinct \emph{parallel edges}.  A loop
contributes two incidences to the degree of its vertex.  A \emph{tree}
is a connected multigraph containing no cycles in this sense.
\subsection{The free product $C_3 * C_2$}\label{subsec:free-product}

For the matching arguments we fix the \emph{abstract} groups
\[
   P:=C_3=\{1,a,a^2\},\qquad Q:=C_2=\{1,b\}\]
and
   \[\Pi:=P*Q=\langle a,b\mid a^3=b^2=1\rangle.
\]
We use the reduced-word model of the free product $\Pi$: a nonidentity
word is a finite nonempty sequence of nonidentity elements of $P$
and $Q$, with adjacent elements from different factors.  Multiplication
is concatenation followed by multiplication within factors and deletion
of identities until the word is reduced.  Every element has one such
reduced form.  Thus permutations $p^3=q^2=\id$ give a $\Pi$-action
by $a\mapsto p$, $b\mapsto q$; this action need not be faithful.
The groups $P,Q$ always mean these abstract cyclic factors.  In particular, freeness of the
$P$-action means three-point orbits, even when an image permutation
might otherwise have order one.

\subsection{Almost finiteness and crossed products}
\label{subsec:almost-finite-prelim}
For subsets $A,B\subseteq X$, write
$
   A\prec B
$
if every closed set $C\subseteq A$ admits a finite open cover
$U_1,\ldots,U_m$ and elements
$\gamma_1,\ldots,\gamma_m\in\Gamma$ such that the sets
$
   \gamma_iU_i,$ $1\le i\le m,
$
are pairwise disjoint subsets of $B$.
The action has \emph{open-set comparison} if
\[
   \mu(A)<\mu(B)
   \qquad
   \text{for every }\mu\in \cM_\Gamma(X)
\]
implies $A\prec B$ for all nonempty open sets $A,B\subseteq X$.

 In the
zero-dimensional setting the comparison and open-set comparison are equivalent
\cite[Proposition~3.6]{Kerr}.

\begin{lemma}\label{lem:clopen-open-comparison}
For an action of a countably infinite discrete group on a compact
metrizable zero-dimensional space, clopen dynamical comparison is
equivalent to open-set comparison.
\end{lemma}

Following \cite[Definition~6.1]{Lindenstrauss1999},  the action has the \emph{small boundary property} (SBP) if
every point has arbitrarily small open neighborhoods $U$ such that
\[
   \mu(\partial U)=0
   \qquad
   \text{for every }\mu\in \cM_\Gamma(X).
\]
  In particular, every action on a
zero-dimensional compact space has SBP, since
the space has a clopen basis.

For a free action, a \emph{tower} is a pair $(V,S)$, where
$V\subseteq X$ and $S\subseteq\Gamma$ is finite and nonempty, such
that the sets
$
   sV,$ $s\in S,
$
are pairwise disjoint.  The set $V$ is the \emph{base}, $S$ is the
\emph{shape}, and the sets $sV$ are the \emph{levels}.  A
\emph{castle} is a finite collection of towers whose levels are
pairwise disjoint.

Given a finite set $K\subseteq\Gamma$ and $\delta>0$, a nonempty
finite set $S\subseteq\Gamma$ is \emph{$(K,\delta)$-invariant} if
$
   |KS\triangle S|<\delta |S|.
$
Following Kerr \cite[Definition~8.2]{Kerr}, a free action is \emph{almost finite} if for every $n\in\mathbb N$,
every finite $K\subseteq\Gamma$ containing $1_\Gamma$, and every
$\delta>0$, there exist an open castle
$
   \{(V_i,S_i)\}_{i\in I}
$
and subsets $S_i'\subseteq S_i$ such that
\begin{enumerate}[label=\textup{(\roman*)}]
\item each $S_i$ is $(K,\delta)$-invariant;
\item every level $sV_i$ has diameter less than $\delta$;
\item $|S_i'|<|S_i|/n$;
\item
$
   X\setminus\bigsqcup_{i\in I}S_iV_i
   \ \prec\
   \bigsqcup_{i\in I}S_i'V_i.
$
\end{enumerate}
\begin{theorem}[{\cite[Theorem~10.2]{Kerr}}]\label{thm:kerr}
    For free actions on zero-dimensional compact metrizable spaces,
almost finiteness is equivalent to the following condition: for every
finite set $K\subseteq\Gamma$ containing $1_\Gamma$ and every $\delta>0$, there is a
clopen castle whose shapes are $(K,\delta)$-invariant and whose
levels partition $X$.
\end{theorem}

\vspace{.3cm}

In the next theorem,
$C(X)\rtimes_r\Gamma$ denotes the reduced crossed product for the induced action \[ (\gamma\cdot f)(x):=f(\gamma^{-1}x) \qquad (f\in C(X),\ \gamma\in\Gamma), \] and $\mathcal Z$ denotes the Jiang--Su algebra. 
Thus, $\mathcal Z$-stability means \[ \bigl(C(X)\rtimes_r\Gamma\bigr)\otimes_{\min}\mathcal Z \cong C(X)\rtimes_r\Gamma. \] For background on reduced crossed products and $\mathcal Z$-stability, see, for example, \cite{BrownOzawa}.

\begin{theorem}\label{thm:known-almost-finite}
We will use  the following two consequences:
\begin{enumerate}[label=\textup{(\roman*)}]
\item
Let a countably infinite discrete amenable group act freely on a
compact metrizable space.  Then the action is almost finite if and only
if it has SBP and open-set comparison
\cite[Theorem~6.1]{KerrSzabo}.

\item
Let a countably infinite discrete group act freely, minimally, and
almost finitely on a compact metrizable space.  Then
$
   C(X)\rtimes_r\Gamma
$
is $\mathcal Z$-stable
\cite[Theorem~12.4]{Kerr}.
\end{enumerate}
\end{theorem}
\subsection{Topologically small boundaries and zero-dimensional extensions}
\label{subsec:topological-small-boundary}

In this subsection, $\Gamma$ is a countably infinite discrete amenable
group, and $\Gamma\acts X$ is an action on a nonempty compact
metrizable space. Write $e=1_\Gamma$. We recall the additional facts
needed to pass from zero-dimensional actions to actions on more general spaces.

A subset $C\subseteq X$ is \emph{topologically small} if there is an
integer $L\geq1$ such that
$
   g_0C\cap\cdots\cap g_LC=\varnothing
$
whenever $g_0,\ldots,g_L\in\Gamma$ are distinct. The action has the
\emph{topological small boundary property} (TSBP) if $X$ has a basis of open
sets with topologically small boundaries
\cite[Definition~7.1]{KerrSzabo}.

We record the following elementary consequences of the definition of
topological smallness; see also \cite[Remark~7.2]{KerrSzabo}.
\begin{lemma}\label{lem:small-facts}
Subsets and finite unions of topologically small sets are topologically
small. Every Borel topologically small set is null for every invariant
probability measure. Moreover, $X$ itself is not topologically small.
Consequently, TSBP implies SBP.
\end{lemma}

The reduction needed below is already available without minimality .
  We record 
\cite[Theorem~7.6 and Corollary~7.7]{KerrSzabo} 
in the
form we shall use.
\begin{theorem}[Kerr--Szab\'o]\label{thm:ks-reduction}
Suppose that every free action of $\Gamma$ on a nonempty compact
metrizable  zero-dimensional   space is almost finite. Then every free
action of $\Gamma$ on a nonempty compact metrizable space with TSBP
is almost finite. In particular, this holds for every free action
on a nonempty finite-dimensional compact metrizable space.
\end{theorem}

\begin{proposition}[{\cite[Theorem~3.8]{SzaboRokhlin}}]
\label{prop:external-finite-dim}
Every free action of a countably infinite discrete group on a nonempty
finite-dimensional compact metrizable space has TSBP.
\end{proposition}

\section{Clopen matching constructions}\label{sec:matching}
The purpose of this section is to isolate the combinatorial mechanism
behind Theorem~\ref{thm:main-three-two}.  Let us first describe the
idea of the proof.  A witness to
\[
   3a\leq 2b
\]
places three disjoint copies of a representative of $a$ inside two
copies of a representative of $b$.  In the application below, the
three copies will be denoted by
$
   F_0,\ F_1,\ F_2
$
and will lie in
$
   W=B^0\sqcup B^1,
$
where $B^0$ and $B^1$ are two copies of a representative of $b$.
Since the three sets $F_i$ represent the same type, they can be
cyclically identified.  This gives an order-three map $p$ on
\[
   F:=F_0\sqcup F_1\sqcup F_2,
\]
whose orbits consist of the three corresponding copies of the same
point.  On the other hand, the two ambient copies $B^0$ and $B^1$
are naturally exchanged by an involution $q$.

The key point is then to find a clopen set $C\subseteq F$ which
meets every $p$-orbit in exactly one point and every $q$-orbit in
at most one point.  The first condition guarantees that $C$ is
equidecomposable with one of the three copies $F_i$, and hence
\[
   [C]=a.
\]
The second condition allows the part of $C$ lying in $B^1$ to be
moved by $q$ into $B^0$ without creating collisions.  More precisely,
\[
   D:=(C\cap B^0)\sqcup q(C\cap B^1)
\]
is then a clopen subset of $B^0$ equidecomposable with $C$.  Thus
\[
   a=[C]=[D]\leq[B^0]=b,
\]
which is exactly the desired conclusion.

This reduces the proof to a clopen selection problem.  It is naturally
encoded as a bipartite matching problem: the left vertices are the
three-point $p$-orbits, the right vertices are the two-point $q$-orbits,
and the points of $F$ are the edges joining the corresponding orbit
classes.  A clopen left-perfect matching is precisely a clopen choice
of one representative from each $p$-orbit, with no $q$-orbit used
twice.

This section develops the matching machinery needed to
produce such a choice.  We first describe the orbit-incidence graph,
then show how to eliminate bounded-length augmenting paths by clopen
modifications.  This leads to a finite criterion for the existence of
a clopen left-perfect matching.  Finally, we verify this criterion for
the orbit graphs arising from the free product
$
   C_3*C_2,
$
which is the algebraic structure generated by the order-three map
$p$ and the involution $q$.  These results will be applied in
Section~\ref{sec:bisection-matching} to the finite-order bisections
arising from the original comparison problem.

\subsection{Orbit incidence graphs}\label{subsec:incidence}
We now encode the preceding selection problem as a bipartite
matching problem.  The set $F$ is the region in which one chooses
representatives of the three-point $p$-orbits, while the
two-point $q$-orbits provide the constraint that no two chosen
points may belong to the same $q$-orbit.  Thus the natural incidence
graph has $p$-orbits in $F$ on the left and $q$-orbits in $Y$ on the
right, with points of $F$ as edges.  The next lemma makes this
correspondence precise.

\begin{lemma}\label{lem:finite-sections}
Let $Y$ be a compact zero-dimensional space, and let
$p,q\in\Homeo(Y)$ satisfy
$
   p^3=q^2=\id_Y.
$
Let $F\subseteq Y$ be clopen and $p$-invariant.  Suppose that every
$p$-orbit in $F$ has three points and every $q$-orbit in $Y$
has two points.
Consider the bipartite multigraph whose left vertex space, right vertex
space, and edge space are
\[
   L=F/\langle p\rangle,\qquad
   R=Y/\langle q\rangle,\qquad
   E=F,
\]
respectively, with endpoint maps
\[
   s(y)=\langle p\rangle y,\qquad
   r(y)=\langle q\rangle y.
\]
Here $F/\langle p\rangle$ and $Y/\langle q\rangle$ denote the
orbit spaces with the quotient topology.
Then this graph has a finite clopen presentation.

Moreover, for a clopen set $C\subseteq F$, the edge subset $C$ is
a left-perfect clopen matching if and only if
\[
   F=C\sqcup pC\sqcup p^2C,
   \qquad
   C\cap qC=\varnothing.
\]
\end{lemma}

\begin{proof}
By Lemma~\ref{lem:finite-group-domain}, choose clopen fundamental
domains $B\subseteq F$ and $D\subseteq Y$ such that
\[
   F=B\sqcup pB\sqcup p^2B
   \qquad\text{and}\qquad
   Y=D\sqcup qD.
\]
The quotient maps restrict to homeomorphisms
\begin{equation}\label{eq:indetifications}
      B\cong F/\langle p\rangle\qquad\text{and}
   \qquad
   D\cong Y/\langle q\rangle. 
\end{equation}
We therefore identify the left and right vertex spaces with separate
copies of $B$ and $D$, respectively.

For $0\le i<3$ and $0\le j<2$, set
\[
   A_{ij}:=p^iB\cap q^jD.
\]
These sets form a clopen partition of $F$.  Under the identifications
\eqref{eq:indetifications},
for an edge $y\in A_{ij}$, the left endpoint $\langle p\rangle y$ is represented by
$p^{-i}y\in B$, while the right endpoint $\langle q\rangle y$
is represented by $q^{-j}y\in D$.
For the nonempty $A_{ij}$,  define
\[
   \phi_{ij}:=q^{-j}p^i:
  p^{-i}A_{ij}\to q^{-j}A_{ij}.
\]
Let
\[
   \widetilde E
   :=
   \bigsqcup_{i,j}
   \bigl(\{(i,j)\}\times p^{-i}A_{ij}\bigr).
\]
Then
\[
   \Psi:\widetilde E\to F,
   \qquad
   \Psi((i,j),x)=p^ix,
\]
is a homeomorphism.  Under the identifications
\eqref{eq:indetifications},
it preserves the source and range maps.  Hence
$\{\phi_{ij}\}$, together with $\Psi$, gives a finite clopen
presentation of the original bipartite multigraph.

Now let $C\subseteq F$ be clopen.  The source condition for a matching
means that $C$ meets each $p$-orbit in at most one point, while
left-perfectness means that it meets each such orbit exactly once.
Thus
\[
   F=C\sqcup pC\sqcup p^2C.
\]
The range condition means that $C$ meets each $q$-orbit in at most
one point, which, since every $q$-orbit has two points, is equivalent
to
$
   C\cap qC=\varnothing.
$
\end{proof}

\subsection{Eliminating short augmenting paths}\label{subsec:augmentations}

Having encoded the selection problem as a clopen matching problem, we
now develop the basic augmentation procedure needed to enlarge a
partial matching.  As in the classical matching argument, the key is
to eliminate short augmenting paths; here the modifications must
remain clopen throughout.  Related Borel constructions appear in
\cite[Proposition~1.1]{ElekLippner} and
\cite[Lemma~2.5]{ConleyJacksonKerrMarksSewardTuckerDrob}.
We include the argument because the clopen regularity required here
does not follow directly from those results.
\begin{lemma}\label{lem:shortest-augmentations}
For every clopen matching $E_{M_0}$ in a graph with a finite clopen
presentation, there is a sequence of clopen matchings $E_{M_h}$,
starting with $E_{M_0}$, such that
\begin{enumerate}[label=\textup{(\roman*)}]
\item $\dom M_h\subseteq\dom M_{h+1}$ and
      $\ran M_h\subseteq\ran M_{h+1}$;
\item for $h\geq1$, $M_h$ has no augmenting path of length at most
      $2h-1$.
\end{enumerate}
\end{lemma}

\begin{proof}
Fix a clopen matching $E_M$.  Direct every graph edge from $L$ to
$R$, and add the matching edges in the reverse direction via
\[
   M^{-1}:\ran M\longrightarrow\dom M.
\]
A directed walk from an unmatched left vertex to an unmatched right
vertex can be shortened, by deleting subwalks between repeated
vertices, to a path.  Such a path is augmenting.  Indeed, whenever the
path reaches an internal left vertex, the preceding right vertex is
its matching partner, and hence the next forward edge cannot be the same
matching edge, since this would repeat the preceding vertex.

Suppose that $M$ has no augmenting path strictly shorter than
$
   \ell=2k+1.
$
Let $I$ be the finite index set of the clopen presentation.  For
$
   \mathbf i=(i_0,\ldots,i_k)\in I^{k+1},
$
let $O_{\mathbf i}\subseteq L$ be the set of unmatched left vertices
for which
\[
   \phi_{i_0},M^{-1},\phi_{i_1},M^{-1},
   \ldots,M^{-1},\phi_{i_k}
\]
is successively defined and ends at an unmatched right vertex.
Since all relevant domains and ranges are clopen,
$O_{\mathbf i}$ is clopen and may be empty.  We simply
discard those strings for which $O_{\mathbf i}=\varnothing$.

For $0\leq j\leq\ell$, let
$
   v_j^{\mathbf i}:O_{\mathbf i}\to L\sqcup R
$
denote the $j$-th vertex of the corresponding walk.  Each
$v_j^{\mathbf i}$ is a homeomorphism onto a clopen image.
Every such walk is a path, otherwise deleting a repeated-vertex
subwalk would give an augmenting path of length smaller than $\ell$.

Let
$
   d(v)\in\mathbb Z_{\geq0}\cup\{\infty\}
$
be the directed distance, with respect to $M$, from the set of
unmatched left vertices.  Along every path above,
\begin{equation}\label{eq:shortest-distance}
   d\bigl(v_j^{\mathbf i}(x)\bigr)=j
   \qquad
   (0\leq j\leq\ell).
\end{equation}
Indeed, a shorter directed path to $v_j^{\mathbf i}(x)$, followed by
the remaining part of the displayed path, would give an augmenting
walk of length smaller than $\ell$, and hence a shorter augmenting
path.

It follows that, for a fixed $\mathbf i$, paths with distinct roots
are vertex-disjoint.  If they met in the same position, injectivity of
$v_j^{\mathbf i}$ would identify their roots; if they met in
different positions, \eqref{eq:shortest-distance} would give two
different values for the same distance.

Enumerate the finitely many strings $\mathbf i\in I^{k+1}$.
Proceeding in this order, retain from each $O_{\mathbf i}$ only
those roots whose paths avoid all vertices used by previously retained
paths.  If $W\subseteq L\sqcup R$ is the previously used vertex set,
the retained roots are
\[
   O_{\mathbf i}'
   =
   O_{\mathbf i}\setminus
   \bigcup_{j=0}^{\ell}
      (v_j^{\mathbf i})^{-1}(W).
\]
Inductively, $W$ and hence $O_{\mathbf i}'$ are clopen.  We obtain
a clopen family of pairwise vertex-disjoint augmenting paths of length
$\ell$, maximal among all such paths.

Flip all retained paths simultaneously and denote the resulting
matching by $M'$. 
Since there are only finitely many strings and
positions, the removed and added edge sets are finite unions of
clopen subsets of the edge charts.  Hence $E_{M'}$ is clopen.
Vertex-disjointness shows that it is a matching.  Moreover, as flipping along an augmenting path preserves the matched status of every previously matched vertex and matches the two previously unmatched endpoints, one has
\begin{equation}\label{eq:incresing}
       \dom M\subseteq\dom M',
   \qquad
   \ran M\subseteq\ran M'.
\end{equation}

It remains to show that $M'$ has no augmenting path of length at most
$\ell$.  Keep the old distance $d$.  Every forward edge and every
unchanged backward matching edge increases $d$ by at most one,
whereas a new backward matching edge reverses a forward edge of a
selected path and, by \eqref{eq:shortest-distance}, decreases $d$
by one.

Suppose that $M'$ has an augmenting path
$
   u_0,u_1,\ldots,u_m
$
of length $m\leq\ell$.
According to \eqref{eq:incresing},
the endpoints $u_0$ and $u_m$, which are unmatched for $M'$,
were already unmatched for $M$.  Hence
$
   d(u_0)=0.
$
Moreover,
$
   d(u_m)\geq\ell,
$
for otherwise a directed path for $M$ from an unmatched left vertex
to $u_m$ of length less than $\ell$ would yield an augmenting path
for $M$ shorter than $\ell$.

Along every directed edge $u_j\to u_{j+1}$ for $M'$, we have
\[
   d(u_{j+1})\leq d(u_j)+1.
\]
Therefore
\[
   \ell
   \leq d(u_m)
   \leq m
   \leq\ell.
\]
Thus $m=\ell$, and equality must hold at every step, i.e.
\[
   d(u_{j+1})=d(u_j)+1
   \qquad(0\leq j<m).
\]
In particular, the path cannot use a new backward matching edge,
since every such edge decreases $d$ by one.

We claim that the hypothetical augmenting path is disjoint from all
retained paths.  Indeed, every right vertex of a retained path is
matched in $M'$.  Hence it cannot be the terminal vertex of the
hypothetical path; if it occurs as an internal vertex, the path must
leave it along its $M'$-matching edge backwards.  This is one of the
new backward matching edges, which we have just shown cannot occur.
Similarly, every left vertex of a retained path is matched in $M'$.
Hence it cannot be the initial vertex of the hypothetical path; if it
occurs later, the path must enter it along a backward $M'$-matching
edge, again one of the new backward matching edges.  Thus the
hypothetical path is disjoint from every retained path.

Moreover, since it uses no new backward matching edge, every directed
edge of the hypothetical path was already present in the directed
graph associated to $M$.  Its endpoints were also unmatched for
$M$.  Hence it is an augmenting path for $M$ of length $\ell$.
This path is disjoint from the retained family, contradicting the
maximality of that family.

Starting with $M_0$, apply this construction successively for
$
   \ell=1,3,5,\ldots.
$
At the $h$-th stage take $\ell=2h-1$.  The resulting sequence
$M_h$ satisfies \textup{(i)} and \textup{(ii)}.
\end{proof}

The preceding lemma gives the following consequence.  The hypothesis
below gives, for each initially unmatched left vertex $x$, a bound
on the length of an augmenting path whenever $x$ is still unmatched.
Compactness will then give a finite stage at which every left vertex
is matched.

\begin{proposition}\label{prop:finite-cores}
Let $E_{M_0}\subset E$ be a clopen matching in a bipartite
multigraph with a finite clopen presentation.  Suppose that for every
$x\in L\setminus\dom M_0$ there are a finite set
$S_x\subseteq L$, with $x\in S_x$, and edges
$
   e_z\in E,$ $z\in S_x,
$
such that $s(e_z)=z$ and
\begin{enumerate}[label=\textup{(\roman*)}]
\item the right endpoints $r(e_z)$, $z\in S_x$, are pairwise
      distinct;
\item every edge whose right endpoint is one of the $r(e_z)$ has
      its left endpoint in $S_x$.
\end{enumerate}
Then there is a clopen left-perfect matching  $E_M\subset E$.

The final matching need not contain the edges of $E_{M_0}$, but
every vertex matched by $M_0$ remains matched.  
\end{proposition}
\begin{proof}
Fix $x\in L\setminus\dom M_0$ and the corresponding finite data.
Define
\[
   m_x:S_x\to R,
   \qquad
   z\mapsto r(e_z).
\]
By \textup{(i)}, $m_x$ is injective.

Let $M$ be any matching for which $x$ is still unmatched.
Starting from $x$, follow the prescribed edge $e_x$ to $m_x(x)$.
If this right vertex is unmatched, we stop.  Otherwise follow its
matching edge backwards and repeat the same procedure.  By
\textup{(ii)}, every left vertex reached in this way belongs to
$S_x$.

The successive left vertices are obtained by iterating the partial map
\[
   T:=M^{-1}\circ m_x,
   \qquad
   \dom T=\{z\in S_x:m_x(z)\in\ran M\}.
\]
The map $T$ is injective.  Moreover,
$
   \ran T\subseteq\dom M,
$
whereas $x\notin\dom M$.  Hence $x\notin\ran T$.

The forward orbit of $x$ under $T$ cannot repeat.  Indeed, if
$
   T^i(x)=T^j(x)$
for $0\le i<j$,
then injectivity gives
\[
   x=T^{j-i}(x)\in\ran T,
\]
a contradiction.  Since all these left vertices lie in the finite set
$S_x$, the procedure must terminate after at most $|S_x|$ left
vertices, at an unmatched right vertex.

The left vertices on the resulting chain are therefore distinct, 
and (i) implies that the corresponding right vertices are distinct as well. 
The prescribed forward edges lie outside the matching. 
Indeed, the first such edge starts at the unmatched vertex
 $x=x_0$. For $i\ge1$,
  the unique matching edge incident to $x_i$ is the edge joining $x_i$ to the preceding right vertex $y_{i-1}$, 
  since $M(x_i)=y_{i-1}$. 
  On the other hand, the prescribed forward edge joins $x_i$ to $y_i=m_x(x_i)$. 
  Since the right vertices are distinct, $y_i\ne y_{i-1}$, 
  and hence this prescribed edge does not belong to $M$. 
  Thus, the chain alternates between edges outside $M$ and edges in $M$, 
  and is therefore an augmenting path of length at most $2|S_x|-1$.
In particular, this bound holds for every matching which still leaves
$x$ unmatched.

Now apply Lemma~\ref{lem:shortest-augmentations}, and set
\[
   U_h:=L\setminus\dom M_h.
\]
The sets $U_h$ are decreasing and clopen.  Fix $x\in U_0$ and put
$
   N_x:=|S_x|.
$
If $x\in U_h$ for some $h\ge N_x$, then the preceding argument,
applied to $M_h$, gives an augmenting path of length at most
$
   2N_x-1\le2h-1,
$
contrary to Lemma~\ref{lem:shortest-augmentations}.  Hence every
$x\in U_0$ is absent from $U_h$ for all sufficiently large $h$,
and therefore
\[
   \bigcap_{h\ge0}U_h=\varnothing.
\]

By compactness of the nested sets $U_h$, we have
$U_H=\varnothing$ for some finite $H$. Hence $M_H$ is a clopen left-perfect matching.  The
monotonicity in Lemma~\ref{lem:shortest-augmentations} also shows that
every vertex matched by $M_0$ remains matched.
\end{proof}

\subsection{Finite data from a relation}
\label{subsec:cycle-exit}

We now verify the hypothesis of
Proposition~\ref{prop:finite-cores} for the orbit-incidence graph
associated with $P=C_3$ and $Q=C_2$.  The argument is based on a
simple dichotomy.  Starting from a point fixed by a nonidentity
element of
$
   \Pi=P*Q,
$
either its orbit component exits the region $F$, in which case a
shortest exit path provides the required finite configuration, or the
entire component remains inside $F$, in which case the nontrivial
relation produces a cycle.  The cycle, together with a path from the
given point to it, again yields the finite configuration required by
Proposition~\ref{prop:finite-cores}.

\begin{lemma}
\label{lem:cycle-exit}
Let $\Pi$ act on a set $Y$.  Suppose that $F\subseteq Y$ is
$P$-invariant, that $P$ acts freely on $F$ and fixes
$Y\setminus F$ pointwise, and that $Q$ acts freely on $Y$.
Consider the bipartite multigraph
\[
   L=F/P,\qquad R=Y/Q,\qquad E=F,
\]
with
\[
   s(y)=Py,\qquad r(y)=Qy.
\]
If $x\in F$ is fixed by some nonidentity $w\in\Pi$, then there
exist a finite set $S\subseteq L$, with $Px\in S$, and edges
$
   \{e_V\}_{V\in S}\subset F$
such that
\begin{enumerate}[label=\textup{(\roman*)}]
\item $s(e_V)=V$ for every $V\in S$, and the right endpoints
      $r(e_V)$, $V\in S$, are pairwise distinct;
\item if $e\in F$ satisfies
    $
         r(e)=r(e_V)
      $
      for some $V\in S$, then $s(e)\in S$.
\end{enumerate}
\end{lemma}

\begin{proof}
Consider first the larger bipartite graph
$\widetilde{\mathcal B}$ with left vertex set $Y/P$, right vertex
set $Y/Q$, and edge set $Y$, where $y\in Y$ joins $Py$ to
$Qy$.

Since $Q$ acts freely, every right vertex $Qy$ is incident to
exactly the two edges $y$ and $by$.  We suppress each such
degree-two right vertex: the two-edge path
\[
   Py \;-\; Qy \;-\; P(by)
\]
is replaced by a single edge, labelled $Qy$, joining $Py$ to
$P(by)$.  In this way we obtain a multigraph $\mathcal H$ with
\[
   V(\mathcal H)=Y/P,
   \qquad
   E(\mathcal H)=Y/Q.
\]
If $Py=P(by)$ the resulting edge is a loop, and distinct
$Q$-orbits may give parallel edges.

The connected component\footnote{The connected
component of a vertex is the set of all vertices that can be joined to
it by a finite path.} of $Px$ in $\mathcal H$ is precisely
\[
   \{P(gx):g\in\Pi\}.
\]

We distinguish two cases.

\smallskip
\noindent
\textbf{Case 1. The component of $Px$ meets $(Y\setminus F)/P$.}
Choose a shortest path
\[
   V_0=Px,V_1,\ldots,V_m,
\]
with
\[
   V_0,\ldots,V_{m-1}\in F/P
   \qquad\text{and}\qquad
   V_m\notin F/P,
\]
and let $\eta_j$ be the edge joining $V_j$ to $V_{j+1}$.
Set
\[
   S:=\{V_0,\ldots,V_{m-1}\}.
\]

For each $j<m$, regard $\eta_j\in Y/Q$ as the corresponding
two-point $Q$-orbit.  Since $\eta_j$ joins the distinct vertices
$V_j$ and $V_{j+1}$, there is a unique point
\[
   e_{V_j}\in \eta_j\cap V_j.
\]
As $V_j\subseteq F$, we have $e_{V_j}\in F$, and in the original
bipartite graph
\[
   s(e_{V_j})=V_j,\qquad r(e_{V_j})=\eta_j.
\]
Since the path edges $\eta_j$ are distinct, the right endpoints
$r(e_{V_j})$ are pairwise distinct. Thus (i) holds.

It remains to verify condition \textup{(ii)}.
Let $e\in F$ satisfy
$
   r(e)=r(e_{V_j})=\eta_j
$
for some $j<m$.  Then $e$ belongs to the $Q$-orbit
$\eta_j$, so its left endpoint $s(e)=Pe$ is one of the two
$P$-endpoints of $\eta_j$, namely $V_j$ and $V_{j+1}$.
If $j<m-1$, both vertices belong to $S$.  If $j=m-1$, then
$V_m\notin F/P$, whereas $e\in F$ implies $Pe\in F/P$.
Hence in this case $s(e)=V_{m-1}\in S$.  Thus condition
\textup{(ii)} holds.

\smallskip
\noindent
\textbf{Case 2. The component of $Px$ is contained in $F/P$.}
Then every point of $\Pi x$ lies in $F$, so both $P$ and $Q$
act freely on $\Pi x$.

Choose a nonidentity element $w\in\Pi$ with $wx=x$, and write
$
   w=\sigma_n\cdots\sigma_1
$
in reduced form.  Put
\[
   x_0=x,\qquad
   x_i=\sigma_i x_{i-1}\quad(1\le i\le n).
\]
Thus $x_n=x_0$.

Pass to the subdivision of $\widetilde{\mathcal B}$ obtained by
inserting one new vertex in the interior of each edge. More precisely,
the edge labelled by $y\in Y$, joining $Py$ to $Qy$, is replaced by
the two-edge path
\[
   Py\;-\;y\;-\;Qy,
\]
where the new subdivision vertex is again denoted by $y$.

Thus the vertices of the subdivided graph consist of the orbit
vertices in $Y/P$ and $Y/Q$, together with one subdivision vertex
for each point of $Y$. For $1\le i\le n$, set
\[
   c_i:=
   \begin{cases}
      Px_{i-1}=Px_i, & \text{if }\sigma_i\in P,\\
      Qx_{i-1}=Qx_i, & \text{if }\sigma_i\in Q.
   \end{cases}
\]
Then $c_i$ is a common orbit vertex adjacent to both $x_{i-1}$ and
$x_i$, so
\[
   x_{i-1},c_i,x_i
\]
is a two-edge path in the subdivided graph. Since $x_n=x_0$,
concatenating these paths gives the nonempty closed walk
\[
   x_0,c_1,x_1,c_2,\ldots,c_n,x_n=x_0.
\]

This walk has no immediate reversal at an internal vertex.  At
$c_i$, the neighbouring vertices $x_{i-1}$ and $x_i$ are
distinct. Indeed,
$
   x_i=\sigma_i x_{i-1},
$
with $\sigma_i\neq1$, and the corresponding factor acts freely.
At $x_i$, $1\le i<n$, the neighbouring vertices are $c_i$ and
$c_{i+1}$.  Since $w$ is reduced, $\sigma_i$ and
$\sigma_{i+1}$ belong to different factors, so one of
$c_i,c_{i+1}$ is a $P$-vertex and the other is a $Q$-vertex.
Hence they are distinct.

On the other hand, a nonempty closed walk in a tree must have an immediate reversal at
some internal occurrence.  Indeed, root the tree at $x_0$ and
choose an occurrence in the walk of a vertex at maximal positive
distance from $x_0$.  The two adjacent vertices in the walk must
both be the unique neighbour of this vertex lying closer to $x_0$,
and hence they coincide.  Therefore this component is not a tree and
contains a cycle.  Undoing the subdivision gives a cycle in
$\widetilde{\mathcal B}$.  Suppressing its $Q$-vertices then gives
a cycle $\mathcal C$ in $\mathcal H$, possibly a loop or a pair
of parallel edges.

Orient $\mathcal C$ and select one outgoing edge at each vertex of
$\mathcal C$.  If $Px\notin V(\mathcal C)$, choose a shortest path
\[
   W_0=Px,W_1,\ldots,W_t,
   \qquad W_t\in V(\mathcal C),
\]
from $Px$ to $\mathcal C$.  Thus
$W_0,\ldots,W_{t-1}\notin V(\mathcal C)$.  For each $j<t$, select
the edge joining $W_j$ to $W_{j+1}$.  Put
\[
   S:=
   V(\mathcal C)\cup\{W_0,\ldots,W_{t-1}\},
\]
where the second set is omitted if $Px\in V(\mathcal C)$.

Thus each $V\in S$ has a selected edge $\eta_V\in Y/Q$ of
$\mathcal H$.  These selected edges are pairwise distinct, and both
endpoints of every $\eta_V$ belong to $S$.  Since the whole
component lies in $F/P$, every $V\in S$ is contained in $F$.
Choose
$
   e_V\in \eta_V\cap V.
$
Then $e_V\in F$, and in the original bipartite graph
\[
   s(e_V)=V
   \qquad\text{and}\qquad
   r(e_V)=\eta_V.
\]
Since the $\eta_V$ are pairwise distinct, so are the right endpoints
$r(e_V)$.

Finally, suppose that $e\in F$ satisfies
$
   r(e)=r(e_V)=\eta_V
$
for some $V\in S$.  Then $e$ and $e_V$ belong to the same
$Q$-orbit, so $Pe$ is one of the endpoints of the edge
$\eta_V$ in $\mathcal H$.  Both endpoints of $\eta_V$ belong
to $S$, and hence
$
   s(e)=Pe\in S.
$
This proves the required conditions.
\end{proof}

\section{Matching for finite-order bisections}
\label{sec:bisection-matching}
The preceding section reduced the required selection to a clopen
matching problem for the orbit-incidence graph associated with
$P=C_3$ and $Q=C_2$, and developed the combinatorial tools needed
to solve it. We now realize this matching problem for finite-order
bisections arising from an amenable group action. This provides the
bridge between the abstract clopen matching argument and the
three-to-two comparison theorem.

Keep the notation $P,Q$ and $\Pi=P*Q$ from
Subsection~\ref{subsec:free-product}. The proof naturally separates
into two parts. On one part of the induced $\Pi$-action, topological
amenability provides the required clopen transversal. On the
complementary part, nontrivial relations in $\Pi$ yield the finite
configurations supplied by Lemma~\ref{lem:cycle-exit} and required by
Proposition~\ref{prop:finite-cores}. We begin with the topologically
amenable case.

\subsection{A clopen transversal for amenable
$C_3*C_2$-actions}
\label{subsec:reiter-transversal}

The following lemma solves the required selection problem for an
arbitrary topologically amenable $\Pi$-action.

\begin{lemma}\label{lem:reiter-transversal}
Let $\Pi$ act topologically amenably on a compact zero-dimensional
space $Z$, and suppose that the restricted $P$-action is free.
Then there exists a clopen set $C\subseteq Z$ such that
\[
   Z=C\sqcup aC\sqcup a^2C,
   \qquad
   C\cap bC=\varnothing.
\]
\end{lemma}

\begin{proof}
Assume $Z\neq\varnothing$, and set $\eta=1/20$.
By topological amenability, there is a norm-continuous map
$z\mapsto m_z\in\Prob(\Pi)$ such that
\begin{equation}\label{eq:reiter-error}
   \sup_{z\in Z}\|m_{hz}-hm_z\|_1<\eta
   \qquad(h\in P\cup Q).
\end{equation}

For $\rho\in\ell^1(\Pi)$ and $h\in\Pi$, write
\[
   (\rho\cdot h)(g):=\rho(gh^{-1}),
\]
and define
\[
   \rho_z(g):=m_z(g^{-1}).
\]

Then by \eqref{eq:reiter-error},
\begin{equation}\label{eq:right-reiter-error}
   \|\rho_{hz}-\rho_z\cdot h^{-1}\|_1<\eta
   \qquad(h\in P\cup Q).
\end{equation}
Define
\[
   \overline\rho_z
   :=\frac13\sum_{a_0\in P}\rho_{a_0z}\cdot a_0.
\]
For $h\in P$,
\begin{equation}\label{eq:exact-P-covariance}
   \overline\rho_{hz}
   =\overline\rho_z\cdot h^{-1}.
\end{equation}
Moreover, by \eqref{eq:right-reiter-error},
\begin{equation*}
   \|\overline\rho_z-\rho_z\|_1<\eta,
\end{equation*}
and hence
\begin{equation}\label{eq:b-error}
   \|\overline\rho_{bz}-\overline\rho_z\cdot b^{-1}\|_1<3\eta.
\end{equation}

Let
\[
   B_1
   :=
   \{1_\Pi\}
   \cup
   \left\{
      g\in\Pi\setminus\{1_\Pi\}:
      g=\sigma_n\cdots\sigma_1
      \text{ in reduced form and }\sigma_1=b
   \right\}.
\]
Every $g\in\Pi$ has a unique decomposition
\[
   g=ua_0,
   \qquad
   u\in B_1,\quad a_0\in P,
\]
so that
\begin{equation}\label{eq:partition-free-product}
   \Pi=\bigsqcup_{a_0\in P}B_1a_0,
\end{equation}
and
\begin{equation}\label{eq:two-cover-free-product}
   B_1\cup B_1b=\Pi.
\end{equation}

Define
\[
   f(z):=\overline\rho_z(B_1).
\]
Since $z\mapsto\overline\rho_z$ is norm-continuous and
\[
   |\rho(B_1)-\sigma(B_1)|
   \leq \|\rho-\sigma\|_1
   \qquad (\rho,\sigma\in\Prob(\Pi)),
\]
the function $f:Z\to[0,1]$ is continuous.

For $a_0\in P$, \eqref{eq:exact-P-covariance} and the definition of
right translation give
\[
   f(a_0z)
   =\overline\rho_{a_0z}(B_1)
   =(\overline\rho_z\cdot a_0^{-1})(B_1)
   =\overline\rho_z(B_1a_0).
\]
Hence, by \eqref{eq:partition-free-product},
\begin{equation}\label{eq:reiter-orbit-sum}
   \sum_{a_0\in P}f(a_0z)
   =
   \sum_{a_0\in P}\overline\rho_z(B_1a_0)
   =
   \overline\rho_z(\Pi)
   =
   1\qquad\text{for all }z\in Z.
\end{equation}
Similarly, \eqref{eq:b-error} gives
\[
   \left|
      f(bz)-(\overline\rho_z\cdot b^{-1})(B_1)
   \right|
   <3\eta,
\]
and therefore
\[
   \overline\rho_z(B_1b)<f(bz)+3\eta.
\]
Using \eqref{eq:two-cover-free-product}, we obtain
\begin{equation}\label{eq:reiter-pair-sum}
   1
   =\overline\rho_z(\Pi)
   \leq
   \overline\rho_z(B_1)+\overline\rho_z(B_1b)
   <
   f(z)+f(bz)+3\eta.
\end{equation}

Choose a clopen $P$-fundamental domain $T\subseteq Z$ by
Lemma~\ref{lem:finite-group-domain}.  For $a_0\in P$, set
\[
   O_{a_0}
   :=\{t\in T:f(a_0t)<1/3+\eta\}.
\]
By \eqref{eq:reiter-orbit-sum}, 
\[
   T=\bigcup_{p\in P}O_p.
\]
Since $T$ is compact and zero-dimensional, there is a clopen
partition
\[
   T=\bigsqcup_{a_0\in P}T_{a_0}
   \qquad\text{with}\qquad
   T_{a_0}\subseteq O_{a_0}.
\]
Set
\[
   C:=\bigsqcup_{a_0\in P}a_0T_{a_0}.
\]
Then $C$ is clopen and meets every $P$-orbit exactly once, so
\[
   Z=C\sqcup aC\sqcup a^2C.
\]
Moreover,
\[
   f(y)<1/3+\eta
   \qquad(y\in C).
\]
If $y,by\in C$, then \eqref{eq:reiter-pair-sum} gives
\[
   1
   <f(y)+f(by)+3\eta
   <\frac23+5\eta
   =\frac{11}{12},
\]
a contradiction.  Hence $C\cap bC=\varnothing$.
\end{proof}

\subsection{The bisection matching theorem}
\label{subsec:bisection-matching-theorem}

We now combine Lemma~\ref{lem:reiter-transversal} with
Lemma~\ref{lem:cycle-exit} and
Proposition~\ref{prop:finite-cores} to obtain the required selection
result for finite-order bisections.  The bisections of orders three
and two induce an action of
$
   \Pi=C_3*C_2
$
on the underlying space.  The proof separates this action into a part
where topological amenability gives the required clopen transversal
and a complementary part where nontrivial relations provide the finite
configurations needed for the matching argument.

\begin{theorem}
\label{thm:bisection-matching}
Let a countably infinite discrete amenable group $K$ act continuously
on a compact metrizable zero-dimensional space $V$, and let
$Y\subseteq V$ be clopen.  Let $p,q$ be compact-open full
bisections of $(K\ltimes V)|_Y$ satisfying
$
   p^3=1_Y,$ $q^2=1_Y,$
where
$
   1_Y=\{(1_K,y):y\in Y\}.
$
Let
\[
   \theta_p=r\circ(s|_p)^{-1}
   \qquad\text{and}\qquad
   \theta_q=r\circ(s|_q)^{-1}
\]
be the homeomorphisms of $Y$ induced by $p$ and $q$.
Suppose that $F\subseteq Y$ is clopen and $\theta_p$-invariant,
every $\theta_p$-orbit in $F$ has three points,
\[
   p\cap s^{-1}(Y\setminus F)=1_{Y\setminus F},
\]
and $\theta_q$ has no fixed point.  Then there exists a clopen set
$C\subseteq F$ such that
\[
   F=C\sqcup\theta_p(C)\sqcup\theta_p^2(C),
   \qquad
   C\cap\theta_q(C)=\varnothing.
\]
\end{theorem}
\begin{proof}
If $F=\varnothing$, there is nothing to prove.  Assume
$F\neq\varnothing$.

Recall that the compact-open full bisections of
$(K\ltimes V)|_Y$ form a group $ \operatorname{Bis}_c((K\ltimes V)|_Y)$ under bisection multiplication. 
The identities
$
   p^3=1_Y,
$ and $
   q^2=1_Y
$
define homomorphisms from $P$ and $Q$, respectively, into this
group, sending $a$ to $p$ and $b$ to $q$.  By the universal
property of the free product, these extend uniquely to a homomorphism
\[
   \Phi:\Pi=P*Q
   \to \operatorname{Bis}_c((K\ltimes V)|_Y).
\]
For $w\in\Pi$, write
$
   B_w:=\Phi(w).
$
Thus
\[
   B_a=p,\qquad B_b=q,
   \qquad B_{uv}=B_uB_v.
\]

Since $B_w$ is a full bisection, for every $y\in Y$, there is a unique arrow of $B_w$ with
source $y$.  Write this arrow in transformation-groupoid notation as
\[
   (\kappa(w,y),y)\in B_w,
   \qquad \kappa(w,y)\in K,
\]
and its range is
\[
  wy:= r(\kappa(w,y),y)=\kappa(w,y)y.
\]
Since $B_{uv}=B_uB_v$, these maps define an action of $\Pi$ on
$Y$, and composition of arrows gives
\begin{equation}\label{eq:kappa-cocycle}
   \kappa(uv,y)
   =
   \kappa(u,vy)\kappa(v,y)
   \qquad(u,v\in\Pi,\ y\in Y).
\end{equation}
In particular,
\[
   ay=\theta_p(y)\qquad\text{and}\qquad
   by=\theta_q(y).
\]

For each fixed $w\in\Pi$, the map
$$
   y\mapsto(\kappa(w,y),y)
   =(s|_{B_w})^{-1}(y)
$$
is continuous. Since $K$ is discrete, 
$
   \kappa(w,\cdot):Y\to K
$
is locally constant.  As $\Pi$ is discrete and $\kappa:\Pi\times Y\to K$ satisfies \eqref{eq:kappa-cocycle},  $\kappa$
is a continuous cocycle.

For $y\in Y$, define the pointwise kernel
\[
   \Lambda_y
   :=
   \{w\in\Pi:\kappa(w,y)=1_K\}.
\]
If $w\in\Lambda_y$, then
$
   wy=\kappa(w,y)y=y.
$
The cocycle identity shows that $\Lambda_y$ is a subgroup of
$\Pi$.  It also gives
\begin{equation}\label{eq:kernel-conjugacy}
   \Lambda_{vy}=v\Lambda_yv^{-1}
   \qquad(v\in\Pi).
\end{equation}
Indeed, if $u\in\Lambda_y$, then $uy=y$, and
\[
   \kappa(vuv^{-1},vy)
   =
   \kappa(v,uy)\,
   \kappa(u,y)\,
   \kappa(v^{-1},vy)  =
   \kappa(v,y)\,
   \kappa(v^{-1},vy)
   =
   1_K.
\]
Thus
$
   v\Lambda_yv^{-1}\subseteq\Lambda_{vy},
$
and the reverse inclusion follows by applying the same argument to
$v^{-1}$.

We now separate the points according to whether the cocycle has a
nontrivial pointwise kernel.  For $w\neq1_\Pi$, put
\[
   E_w
   :=
   \{y\in Y:\kappa(w,y)=1_K\},
\]
and set
\[
   U:=\bigcup_{w\neq1_\Pi}E_w,
   \qquad
   Z:=Y\setminus U.
\]
Each $E_w$ is clopen, since $\kappa(w,\cdot)$ is locally
constant.  Hence $U$ is open and $Z$ is closed.
By \eqref{eq:kernel-conjugacy}, both $U$ and $Z$ are
$\Pi$-invariant.

If $y\in Y\setminus F$, then $p$ agrees with the unit bisection at
$y$.  Since $B_a=p$, this means
$
   \kappa(a,y)=1_K.
$
As $a\neq1_\Pi$,
$
   Y\setminus F\subseteq E_a\subseteq U,
$
and therefore
\begin{equation}\label{eq:Z-contained-F}
   Z\subseteq F.
\end{equation}

By definition of $Z$,
\begin{equation}\label{eq:trivial-kernel-on-Z}
   \kappa(w,z)=1_K
   \quad\Longrightarrow\quad
   w=1_\Pi
   \qquad(w\in\Pi,\ z\in Z).
\end{equation}
If $Z\neq\varnothing$, Lemma~\ref{lem:cocycle-criterion} applied to
the restricted cocycle $\kappa$
shows that the action $\Pi\curvearrowright Z$ is topologically
amenable.  By \eqref{eq:Z-contained-F}, the restricted $P$-action
on $Z$ is free, since every $\theta_p$-orbit in $F$ has three
points.  Lemma~\ref{lem:reiter-transversal} therefore gives a
relatively clopen set $C_Z\subseteq Z$ such that
\begin{equation}\label{eq:CZ-properties}
   Z
   =
   C_Z\sqcup\theta_p(C_Z)\sqcup\theta_p^2(C_Z),
   \qquad
   C_Z\cap\theta_q(C_Z)=\varnothing.
\end{equation}
If $Z=\varnothing$, then
\eqref{eq:CZ-properties} remains valid by letting $C_Z=\varnothing$.

We next extend this matching from $Z$ to a clopen partial matching
on the whole incidence graph.  Since $C_Z$ is clopen in the closed
subspace $Z\subseteq F$, there is a clopen set $C'\subseteq F$ such that
\[
   C'\cap Z=C_Z.
\]
Set
\[
   C_0
   :=
   C'\setminus
   \bigl(
      \theta_p(C')
      \cup\theta_p^2(C')
      \cup\theta_q(C')
   \bigr).
\]
Then
\[
   C_0\cap\theta_p(C_0)
   =
   C_0\cap\theta_p^2(C_0)
   =
   C_0\cap\theta_q(C_0)
   =
   \varnothing.
\]
Applying $\theta_p$ to the first equality also gives
\[
   \theta_p(C_0)\cap\theta_p^2(C_0)=\varnothing.
\]
Hence $C_0$ meets every $P$-orbit and every $Q$-orbit at most
once.
Since $Z$ is $\Pi$-invariant, by \eqref{eq:CZ-properties},
\begin{equation}\label{eq:no shrinking}
       C_0\cap Z
   =
   C_Z\setminus
   \bigl(
      \theta_p(C_Z)
      \cup\theta_p^2(C_Z)
      \cup\theta_q(C_Z)
   \bigr)  =C_Z.
\end{equation}

Consider now the orbit incidence graph from
Lemma~\ref{lem:finite-sections}.  Its left vertex space, right vertex
space, and edge space are
\[
   L=F/P,
   \qquad
   R=Y/Q,
   \qquad
   E=F,
\]
with endpoint maps
\[
   s(x)=Px,
   \qquad
   r(x)=Qx.
\]
Since $C_0$ meets every $P$-orbit and every $Q$-orbit at most
once, the clopen edge set
\[
   E_{M_0}:=C_0
\]
defines a clopen matching $M_0$.  Moreover,   \eqref{eq:CZ-properties}  and \eqref{eq:no shrinking} show that
$M_0$ covers every left vertex represented by a point of $Z$.

It remains to treat the left vertices not covered by $M_0$.
Let $Px\in F/P$ be unmatched, and choose $x\in F$ representing
it.  If $x\in Z$, then the $P$-invariance of $Z$ and
\eqref{eq:CZ-properties} would imply that the orbit $Px$ is already
matched by $M_0$, a contradiction.  Hence $x\notin Z$, so
$x\in U$.  By the definition of $U$, there exists
$w\neq1_\Pi$ such that
$
   \kappa(w,x)=1_K,
$
and hence
\[
   wx=\kappa(w,x)x=x.
\]

The induced $\Pi$-action on $Y$ satisfies the hypotheses of
Lemma~\ref{lem:cycle-exit}.  Indeed, the $P$-action is free on
$F$, it fixes $Y\setminus F$ pointwise as $p$ is the unit
bisection there, and the $Q$-action is free on $Y$ as
$\theta_q$ has no fixed point.  Since $w\neq1_\Pi$ fixes $x$,
Lemma~\ref{lem:cycle-exit} provides, for the unmatched vertex $Px$,
the finite data required by Proposition~\ref{prop:finite-cores}.
  Applying that proposition
gives a clopen left-perfect matching in the orbit incidence graph.
Let $C\subseteq F$ be its edge set.  By
Lemma~\ref{lem:finite-sections},
\[
   F
   =
   C\sqcup\theta_p(C)\sqcup\theta_p^2(C),
   \qquad
   C\cap\theta_q(C)=\varnothing.
\]
This is the required set $C$.
\end{proof}

\section{The three-to-two theorem and algebraic consequences}
\label{sec:three-two}
We now return to the original comparison problem
\[
   3a\leq2b.
\]
The preceding section established the bisection matching theorem,
which provides the clopen selection mechanism needed to pass from
this inequality to $a\leq b$. We first use it to prove
Theorem~\ref{thm:main-three-two}, including its extension to arbitrary
compact zero-dimensional spaces.

We then combine Theorem~\ref{thm:main-three-two} with the clopen
augmentation procedure to prove the  comparison statement
in Theorem~\ref{thm:main-comparison}\,(i) and its
full-group conclusion.  For minimal Cantor actions,   the same
three-to-two law gives weak comparability and hence the algebraic
regularity  in Theorem~  \ref{thm:main-comparison} \textup{(ii)}.
We finally give a nonminimal counterexample to the remaining
structural conclusions.

\subsection{Proof of Theorem~\ref{thm:main-three-two}}
\label{subsec:three-two}

For the metrizable case, the idea is to convert a witness to
$3a\leq2b$ into an order-three bisection $p$ and an order-two
bisection $q$ on two copies of a representative of $b$.
Theorem~\ref{thm:bisection-matching} then provides a clopen set
representing one copy of $a$ which can be moved, without collisions,
into a single copy of $b$.  This gives $a\leq b$.  We then remove
the metrizability assumption by passing to the metrizable
zero-dimensional factor generated by the finitely many clopen sets
appearing in the given equidecomposition.

\begin{proof}[Proof of Theorem~\ref{thm:main-three-two}]
We first assume that $X$ is metrizable.  The case $a=0$ is
immediate.  If $b=0$, then $3a\leq0$, and conicality of
$\Talpha$ gives $a=0$.  Hence we may assume that
$a,b\neq0$.

Choose $A,B\in\Clop_b(\widetilde X)$ with
$[A]=a$ and $[B]=b$.  Let $B^0,B^1$ be disjoint tagged copies of $B$, and put
\[
   W:=B^0\sqcup B^1.
\]
Fix a finite clopen equidecomposition witnessing $3a\leq2b$.
Its three copies of $A$ are mapped onto pairwise disjoint clopen
sets
$
   F_0,F_1,F_2\subseteq W.
$
After retagging the three source copies, we may regard these as three
finite clopen equidecompositions from the same representative $A$
onto $F_0,F_1,F_2$.

To apply Theorem~\ref{thm:bisection-matching}, we realize these
equidecompositions as bisections of a single amenable transformation
groupoid.  After a finite retagging and refinement according to source
and target sheets, all sets involved lie in
$
   X\times\{0,\ldots,d-1\}
$
for some $d\geq1$.  Identify the sheet set with
\[
   C_d:=\mathbb Z/d\mathbb Z
\]
and let
\[
   K:=\Gamma\times C_d,
   \qquad
   (\gamma,t)(x,j):=(\gamma x,j+t).
\]
If a piece sends $(x,j)$ to $(\gamma x,i)$, assign to it the
label
$
   (\gamma,i-j)\in K.
$
These labels are compatible with composition, since
\[
   (\eta,k-i)(\gamma,i-j)
   =
   (\eta\gamma,k-j).
\]
Thus the three equidecompositions determine compact-open bisections
$E_i$ of $K\ltimes(X\times C_d)$ satisfying
\begin{equation}\label{eq:sourse}
       s(E_i)=A
   \qquad\text{and}\qquad
   r(E_i)=F_i
   \qquad(0\leq i<3).
\end{equation}
The group $K$ is amenable, since it is the product of the amenable
group $\Gamma$ with the finite group $C_d$.

By \eqref{eq:sourse},
the inverse bisection $E_0^{-1}$ has source $F_0$ and range
$A$.  Hence
$
   v_i:=E_iE_0^{-1}
$
is a compact-open bisection with
$
   s(v_i)=F_0$ and $r(v_i)=F_i.
$
Its induced homeomorphism is
\[
   f_i:=\theta_{v_i}
   =\theta_{E_i}\circ\theta_{E_0}^{-1}:F_0\to F_i.
\]
In particular,
\[
   v_0=1_{F_0}
   \qquad\text{and}\qquad
   f_0=\operatorname{id}_{F_0}.
\]
Put
\[
   F:=F_0\sqcup F_1\sqcup F_2
\]
and define
\[
   p
   :=
   1_{W\setminus F}
   \sqcup
   \bigsqcup_{i=0}^2v_{i+1}v_i^{-1},
\]
where the subscripts are read modulo $3$.
The summand $v_{i+1}v_i^{-1}$ maps $F_i$ onto $F_{i+1}$.
On $F_i$, three successive applications give
\[
   (v_iv_{i+2}^{-1})
   (v_{i+2}v_{i+1}^{-1})
   (v_{i+1}v_i^{-1})
   =
   v_iv_i^{-1}
   =
   1_{F_i}.
\]
Hence
\[
   p^3=1_W.
\]
Moreover,
\[
   \theta_p(f_i(z))=f_{i+1}(z)
   \qquad(z\in F_0),
\]
so every $\theta_p$-orbit in $F$ has exactly three points, while
$p$ is the unit bisection on $W\setminus F$.

Let
$
   \tau:B^0\to B^1
$
be the canonical retagging which preserves the $X$-coordinate.
If
$
   (x,j_0)\in B^0
$ and $
   \tau(x,j_0)=(x,j_1)\in B^1,$
then $\tau$ is represented in
$K\ltimes(X\times C_d)$ by the arrow with label
$
   (1_\Gamma,j_1-j_0),
$
while $\tau^{-1}$ has the inverse label
$(1_\Gamma,j_0-j_1)$.

Let $q$ be the compact-open full bisection over
$W=B^0\sqcup B^1$ inducing $\tau$ on $B^0$ and
$\tau^{-1}$ on $B^1$.  Then
\[
   q^2=1_W,
\]
and, since $B^0\cap B^1=\varnothing$, the induced homeomorphism
$\theta_q$ has no fixed point.

We may therefore apply Theorem~\ref{thm:bisection-matching} to the
action of $K$ on $X\times C_d$, with $Y=W$ and the clopen
$\theta_p$-invariant set $F$.  It gives a clopen set
$C\subseteq F$ such that
\begin{equation}\label{eq:three-two-C}
   F
   =
   C\sqcup\theta_p(C)\sqcup\theta_p^2(C),
   \qquad
   C\cap\theta_q(C)=\varnothing.
\end{equation}

We now translate this matching back into the type monoid.  For
$i=0,1,2$, set
\[
   D_i:=f_i^{-1}(C\cap F_i).
\]
For each $z\in F_0$, the three points
$
   f_0(z),$ $f_1(z),$ $f_2(z)
$
form one $\theta_p$-orbit.  By
\eqref{eq:three-two-C}, this orbit meets $C$ exactly once.
Consequently,
\[
   F_0=D_0\sqcup D_1\sqcup D_2.
\]
The restrictions
$
   f_i|_{D_i}:D_i\to C\cap F_i
$
give a finite clopen equidecomposition from $F_0$ onto $C$.
Hence
\[
   [C]=[F_0]=a.
\]

Next set
\[
   D
   :=
   (C\cap B^0)
   \sqcup
   \theta_q(C\cap B^1).
\]
Since $\theta_q$ exchanges $B^0$ and $B^1$,
$
   D\subseteq B^0.
$
The identity on $C\cap B^0$, together with
$\theta_q$ on $C\cap B^1$, therefore gives a finite clopen
equidecomposition from $C$ onto $D$.  Thus
\[
   [D]=[C]=a.
\]
Since $D\subseteq B^0$ and $[B^0]=b$, we obtain
\[
   a\leq b.
\]

The auxiliary factor $C_d$ is used only to record changes of sheet
labels.  Every bisection map constructed above has the form
\[
   (x,j)\mapsto(\gamma x,i)
   \qquad(\gamma\in\Gamma),
\]
and hence is an allowed clopen equidecomposition for the original
$\Gamma$-action.  Therefore the inequality obtained above is indeed
$a\leq b$ in $T(\alpha)$.

We now remove the metrizability assumption.  The representatives $A,B$ and the fixed finite equidecomposition
witnessing $3a\leq2b$ involve only finitely many clopen subsets of
$X$ through their finitely many sheet coordinates.  Let
$\mathscr B\subseteq\Clop(X)$ be the $\Gamma$-invariant Boolean
algebra generated by all these clopen sets.  Since
$\Gamma$ is countable, $\mathscr B$ is countable.

Let $X_0$ be the Stone space of $\mathscr B$.  By Stone duality
(see, for example, \cite[Sections~7 and 8]{Koppelberg}), $X_0$ is
compact, metrizable, and zero-dimensional, and the
$\Gamma$-invariance of $\mathscr B$ induces a continuous
$\Gamma$-action on $X_0$ together with an equivariant factor map
$
   \pi:X\to X_0
$
such that every $U\in\mathscr B$ is the inverse image of a clopen
subset of $X_0$.

Hence the chosen representatives and the finite equidecomposition
descend to $X_0\times\mathbb N$, giving types $a_0,b_0$ with
$
   3a_0\leq2b_0.
$
By the metrizable case, $a_0\leq b_0$.  Pulling back a finite clopen
equidecomposition witnessing this inequality gives $a\leq b$.
\end{proof}

 \subsection{Dynamical comparison without minimality}
\label{subsec:nonminimal-comparison}

We now turn from Theorem \ref{thm:main-three-two} to its comparison consequence
without assuming minimality. Throughout this subsection,
$\alpha:\Gamma\acts X$ is a continuous action of a countably infinite
discrete amenable group on a nonempty compact Hausdorff
zero-dimensional space.

The main difficulty is that, without minimality, a nonzero type need
not be an order unit. We overcome this by combining
Theorem~\ref{thm:main-three-two} with the clopen matching procedure of
Lemma~\ref{lem:shortest-augmentations}. Starting from a strict
inequality under all invariant measures, we first obtain a finite
amplified clopen injection. The finitely many group elements occurring
in this injection are then fixed throughout the construction, which
gives uniform control of the remaining target while the matching is
successively improved. This allows the three-to-two theorem to absorb
the unmatched remainder and yields dynamical comparison without
simplicity or cancellation.

\begin{lemma}\label{lem:nm-uniform-averages}
Let $f\in C(X,\mathbb R)$ and suppose that
\[
 \delta:=\min_{\mu\in\cM_\Gamma(X)}\int_X f\,d\mu>0.
\]
Then, for every F\o lner sequence $(F_j)$,
\[
 \frac{1}{|F_j|}\sum_{\gamma\in F_j}f(\gamma x)>\delta/2
 \qquad(x\in X)
\]
for all sufficiently large $j$.

Consequently, if $A,B\in\Clop_b(\widetilde X)$ satisfy
$
 \mu([A])<\mu([B])
$ for all $\mu\in\cM_\Gamma(X)$
then, for some $n\geq1$,
\begin{equation}\label{eq:nm-amplified-comparison}
 n[A]+u\leq n[B].
\end{equation}
\end{lemma}

\begin{proof}
The first assertion is the standard compactness argument for F\o lner
averages. Otherwise, for some $j_k\to\infty$ and $x_k\in X$, the
empirical measures
\[
 \nu_k:=\frac{1}{|F_{j_k}|}
        \sum_{\gamma\in F_{j_k}}\delta_{\gamma x_k}
\]
satisfy $\int f\,d\nu_k\leq\delta/2$. Any weak* convergent subnet has
an invariant limit, by the F\o lner property, contradicting the
definition of $\delta$.

Apply this to $f=m_B-m_A$. For some finite F\o lner set $F$,
\[
 \sum_{\gamma\in F}(m_B-m_A)(\gamma x)\geq1
 \qquad(x\in X),
\]
where we use that the sum is integer-valued. Passing to types and
putting $n=|F|$ gives
\[
 u+n[A]\leq n[B].\qedhere
\]
\end{proof}

\begin{lemma}\label{lem:nm-uniform-remainder}
Let $A,B\in\Clop_b(\widetilde X)$ satisfy
$
 \mu([A])<\mu([B])$ for all $\mu\in\cM_\Gamma(X),$
and fix a finite set $S\subseteq\Gamma$. There is an integer
$L\geq1$ with the following property. Whenever
$M:D\to C$ is a finite clopen equidecomposition between
$D\subseteq A$ and $C\subseteq B$, all of whose group labels
belong to $S$, then
\begin{equation}\label{eq:nm-uniform-target}
 u\leq L[B\setminus C].
\end{equation}
The integer $L$ is independent of $M$ and of its clopen pieces.
\end{lemma}

\begin{proof}
Choose $d\geq1$ such that
$A,B\subseteq X\times[d]$, and put
\[
 f=m_B-m_A,
 \qquad
 \delta:=\min_{\mu\in\cM_\Gamma(X)}
              \int_X f\,d\mu>0.
\]
By Lemma~\ref{lem:nm-uniform-averages},
choose a nonempty finite set $F\subseteq\Gamma$ such that
\begin{equation}\label{eq:nm-fixed-window}
 \sum_{g\in F}f(gx)>\frac{\delta}{2}|F|
 \quad(x\in X),
\qquad\text{and}\qquad
 d\sum_{\gamma\in S}
   |\gamma^{-1}F\mathbin\triangle F|
 <\frac{\delta}{4}|F|.
\end{equation}

Fix such an equidecomposition $M:D\to C$, and write
\[
 D=\bigsqcup_{r=1}^N D_r,
 \qquad
 M(x,i_r)=(\gamma_r x,j_r)
 \quad\text{on }D_r,
\]
where $\gamma_r\in S$. For $\gamma\in S$, set
\[
 D_\gamma:=\bigsqcup_{\{r:\gamma_r=\gamma\}}D_r,
 \qquad
 p_\gamma:=m_{D_\gamma}.
\]
Then
\[
 D=\bigsqcup_{\gamma\in S}D_\gamma,
 \quad
 C=\bigsqcup_{\gamma\in S}M(D_\gamma),
\quad\text{and}\quad
 m_{M(D_\gamma)}(x)
 =
 p_\gamma(\gamma^{-1}x).
\]
Consequently,
\[
 m_D
 =
 \sum_{\gamma\in S}p_\gamma,
 \qquad
 m_C
 =
 \sum_{\gamma\in S}p_\gamma\circ\gamma^{-1},
 \qquad
 0\leq p_\gamma\leq d.
\]
Hence, by \eqref{eq:nm-fixed-window} for every $x\in X$,
\begin{align*}
\left|\sum_{g\in F}(m_C-m_D)(gx)\right|
&\leq
\sum_{\gamma\in S}
\left|
\sum_{h\in\gamma^{-1}F}p_\gamma(hx)
-
\sum_{h\in F}p_\gamma(hx)
\right|\\
&\leq
d\sum_{\gamma\in S}
|\gamma^{-1}F\mathbin\triangle F|
<
\frac{\delta}{4}|F|.
\end{align*}

Put
$
 U=A\setminus D,
$ $
 V=B\setminus C.
$
Since
$
 m_V=f+m_U-(m_C-m_D),
$
we obtain from \eqref{eq:nm-fixed-window} that
\[
 \sum_{g\in F}m_V(gx)
 >
 \frac{\delta}{4}|F|
 >0
 \qquad(x\in X).
\]
The left-hand side is integer-valued, and therefore it is at least
$1$ everywhere. Passing to types gives
\[
 u
 \leq
 \left[\sum_{g\in F}m_V\circ g\right]_T
 =
 |F|[V].
\]
Thus \eqref{eq:nm-uniform-target} holds with $L=|F|$.
\end{proof}

For $n\in\mathbb N$, we regard $A\times[n]$ as $n$ disjoint tagged
copies of $A$. After retagging the sheet coordinates, this is again
viewed as a bounded clopen subset of $\widetilde X$, and hence
\[
   [A\times[n]]=n[A].
\]

Fix $A,B\in\Clop_b(\widetilde X)$ and suppose that
$
   n[A]\leq n[B].
$
Choose a finite clopen equidecomposition from $A\times[n]$ onto a
clopen subset of $B\times[n]$, and denote the resulting injection by
\[
   J:A\times[n]\longrightarrow B\times[n].
\] 
Refining its pieces by the two $[n]$-coordinates, we may
write them in the form
\[
   J(z,k_i)=\bigl(\phi_i(z),\ell_i\bigr),
   \qquad z\in C_i,
\]
where $C_i\subseteq A$ is clopen,
$k_i,\ell_i\in[n]$, and
$
   \phi_i:C_i\to \phi_i(C_i)\subseteq B
$
is a clopen partial homeomorphism given by a group element together
with a change of the original sheet coordinate.

Associated with $J$ is the bipartite multigraph $\mathcal G_J$ with
left and right vertex spaces given by disjoint copies of $A$ and $B$,
and edge space
\[
   E_J:=\bigsqcup_i\bigl(\{i\}\times C_i\bigr),
\]
with endpoint maps
\[
   s(i,z)=z,
   \qquad
   r(i,z)=\phi_i(z).
\]
Thus $\mathcal G_J$ has a finite clopen presentation. The index $i$
is retained, so parallel edges are allowed.

\begin{lemma}\label{lem:nm-packing-unmatched}
Let $M$ be a clopen matching in $\mathcal G_J$, and suppose that
$M$ has no augmenting path of length at most $2q-1$, where $q\geq1$.
If
$
   U:=A\setminus\dom M,
$
then
\begin{equation}\label{eq:nm-packing}
   (q+1)n[U]\leq n[A].
\end{equation}
\end{lemma}

\begin{proof}
The matching $M$ induces a clopen partial homeomorphism
\[
   M:\dom M\to\ran M
\]
from a clopen subset of $A$ onto a clopen subset of $B$.
On its natural domain in $A\times[n]$, define
\[
   T:=(M^{-1}\times\id_{[n]})\circ J.
\]
Then $T$ is an injective clopen partial homeomorphism, piecewise given
by group elements and sheet changes, and
\begin{equation}\label{eq:nm-packing-range}
   \ran T\subseteq(\dom M)\times[n].
\end{equation}

We claim that $T^j$ is defined on all of $U\times[n]$ for
$0\leq j\leq q$. Suppose otherwise, and let
$1\leq r\leq q$ be the first stage at which $T^r$ fails to be
defined at some $z\in U\times[n]$. Then $T^{r-1}(z)$ is defined, but
\[
   J\bigl(T^{r-1}(z)\bigr)
   \notin(\ran M)\times[n].
\]
Write
\[
   z_j:=T^j(z)\qquad(0\leq j\leq r-1),
\]
and let $x_j\in A$ and $y_j\in B$ be the vertices obtained from
$z_j$ and $J(z_j)$, respectively, by discarding the extra
$[n]$-labels. Since $T^{j+1}(z)$ is defined for $j<r-1$,
\[
   y_j\in\ran M,
   \qquad
   x_{j+1}=M^{-1}(y_j).
\]
On the other hand,
\[
   x_0\notin\dom M,
   \qquad
   y_{r-1}\notin\ran M.
\]
Thus
\[
   x_0,y_0,x_1,y_1,\ldots,x_{r-1},y_{r-1}
\]
is a directed walk from an unmatched left vertex to an unmatched
right vertex, consisting alternately of an edge coming from $J$
and a matching edge traversed backwards. Its length is
\[
   2r-1\leq2q-1.
\]
Deleting subwalks between repeated vertices produces an augmenting
path of no greater length, contradicting the hypothesis.

It remains to observe that
$
   T^j(U\times[n]),
$ $0\leq j\leq q,
$
are pairwise disjoint. Indeed, suppose
\[
   T^i(z)=T^j(z')
   \qquad
   (0\leq i<j\leq q),
\]
with $z,z'\in U\times[n]$. Since $T^i$ is injective,
\[
   z=T^{j-i}(z')\in\ran T,
\]
which contradicts \eqref{eq:nm-packing-range}, since
$z\in U\times[n]$ and $U\cap\dom M=\varnothing$.

Each set $T^j(U\times[n])$ has type $n[U]$, and their disjoint union
is contained in $A\times[n]$. Hence
\[
   (q+1)n[U]
   =
   \sum_{j=0}^q [T^j(U\times[n])]
   \leq [A\times[n]]
   =
   n[A],
\]
which proves \eqref{eq:nm-packing}.
\end{proof}
The equivalence of \textup{(i)} and \textup{(ii)} in the following
lemma is \cite[Lemma~2.7]{Melleray}. The orbit characterization
in \textup{(iii)} is implicit in its proof. We include the argument
for completeness.
\begin{lemma}\label{lem:nm-order-unit-test}
For $c\in T(\alpha)$, the following are equivalent:
\begin{enumerate}[label=\textup{(\roman*)}]
\item $c$ is an order unit;
\item $\mu(c)>0$ for every $\mu\in\cM_\Gamma(X)$;
\item for a representative $C\in\Clop_b(\widetilde X)$ of $c$,
      the clopen set $\{x:m_C(x)>0\}$ meets every $\Gamma$-orbit. For a clopen subset of $X$, this is equivalent to the condition $\Gamma C=X$.
\end{enumerate}
\end{lemma}

\begin{proof}
If $c$ is an order unit, then $u\leq kc$ for some $k\geq1$.
Hence, for every invariant probability measure $\mu$,
$
   1=\mu(u)\leq k\mu(c),
$
so $\mu(c)>0$. Thus \textup{(i)} implies \textup{(ii)}.

To prove that \textup{(ii)} implies \textup{(iii)}, put
\[
   W:=\{x\in X:m_C(x)>0\}.
\]
If $\Gamma W\neq X$, then $X\setminus\Gamma W$ is a nonempty closed
invariant subspace. By amenability, it supports an invariant regular
Borel probability measure $\mu$. Viewed as a measure on $X$, this
measure satisfies $\mu(c)=0$, contradicting \textup{(ii)}.

Finally, suppose that $\Gamma W=X$. By compactness, finitely many
translates of $W$ cover $X$, and therefore
$
   u\leq k[W]\leq kc
$
for some $k\geq1$. Since $u$ is an order unit, every
$a\in T(\alpha)$ satisfies $a\leq mu$ for some $m\geq1$, and hence
$
   a\leq mu\leq mkc.
$
Thus $c$ is an order unit, proving \textup{(iii)} implies
\textup{(i)}.
\end{proof}

\begin{proof}[Proof of Theorem~\ref{thm:main-comparison} \textup{(i)}]
Choose representatives $A,B\in\Clop_b(\widetilde X)$ of $a,b$
and $d\geq1$ such that
\[
   A,B\subseteq X\times[d].
\]
If $a=0$, then $\mu(b)>0$ for every
$\mu\in\cM_\Gamma(X)$, so Lemma~\ref{lem:nm-order-unit-test}
shows that $b$ is an order unit. We may therefore assume $a\neq0$.

By Lemma~\ref{lem:nm-uniform-averages}, there are $n\geq1$ and a
finite clopen equidecomposition giving an injection
\[
   J:A\times[n]\to B\times[n].
\]
Let $\mathcal G_J$ be the associated bipartite multigraph, and let
$S\subseteq\Gamma$ be the finite set of group labels occurring in
its edge charts. Lemma~\ref{lem:nm-uniform-remainder} gives
$L\geq1$ such that, for every clopen matching $M$ in $\mathcal G_J$,
\[
   u\leq L[B\setminus\ran M].
\]

Choose $h\geq1$ such that
$
   2^h\geq ndL,
$
and then choose $q\geq1$ such that
$
   (q+1)n\geq3^h.
$
Apply Lemma~\ref{lem:shortest-augmentations} to the empty matching
and let $M$ be the matching obtained after $q$ stages. Then $M$ has
no augmenting path of length at most $2q-1$.

Put
\[
   U:=A\setminus\dom M,
   \qquad
   V:=B\setminus\ran M,
   \qquad
   r:=[U],\quad t:=[V].
\]
By Lemma~\ref{lem:nm-packing-unmatched},
$
   (q+1)nr\leq na,
$
while Lemma~\ref{lem:nm-uniform-remainder} gives
$
   u\leq Lt.
$
Since $A\subseteq X\times[d]$, we also have $a\leq du$. Hence
\begin{equation*}
   3^hr
   \leq(q+1)nr
   \leq na
   \leq nd u
   \leq ndL t
   \leq2^ht.
\end{equation*}
Repeated application of Theorem~\ref{thm:main-three-two} yields
\[
   r\leq t.
\]
Thus $U$ is equidecomposable with a clopen subset  $W\subset V$.
Combining this equidecomposition with the matching $M$ gives an
equidecomposition of
$
   A=\dom M\sqcup U
$
with the clopen subset
$
   \ran M\sqcup W\subseteq B.
$
Hence $a\leq b$. Choose $c\in T(\alpha)$ such that $b=a+c$.
Then
\[
   \mu(c)=\mu(b)-\mu(a)>0
   \qquad(\mu\in\cM_\Gamma(X)),
\]
so Lemma~\ref{lem:nm-order-unit-test} shows that $c$ is an order unit.

It remains to prove the full-group assertion. Let $A,B\subseteq X$
be clopen and satisfy
$
   \mu(A)<\mu(B)
$ for all $\mu\in\cM_\Gamma(X).$
Set
\[
   A_0:=A\setminus B
   \qquad\text{and}\qquad
   B_0:=B\setminus A.
\]
Then
\[
   \mu(A_0)<\mu(B_0)
   \qquad(\mu\in\cM_\Gamma(X)).
\]
By the comparison just proved, there are a clopen set
$C\subseteq B_0$ and a finite clopen equidecomposition
$
   \varphi:A_0\to C.
$
Define
\[
   g(x):=
   \begin{cases}
      \varphi(x),&x\in A_0,\\
      \varphi^{-1}(x),&x\in C,\\
      x,&x\in X\setminus(A_0\cup C).
   \end{cases}
\]
Then $g\in[[\alpha]]$, $g^2=\id_X$, and
\[
   g(A)=(A\cap B)\sqcup C\subseteq B.
\]
Moreover, for every $\mu\in\cM_\Gamma(X)$,
\begin{align*}
   \mu\bigl(B\setminus g(A)\bigr)
   &=\mu(B_0)-\mu(C)\\
   &=\mu(B\setminus A)-\mu(A\setminus B)\\
   &=\mu(B)-\mu(A)>0.
\end{align*}
Hence by Lemma~\ref{lem:nm-order-unit-test},
\[
   \Gamma\bigl(B\setminus g(A)\bigr)=X.\qedhere
\]
\end{proof}

\begin{corollary}\label{cor:nm-direct-consequences}
Let $\alpha:\Gamma\acts X$ be a continuous action of a countably
infinite discrete amenable group on a nonempty compact Hausdorff
zero-dimensional space. Then the following hold.
\begin{enumerate}[label=\textup{(\roman*)}]
\item If $b\in T(\alpha)$ is an order unit, then
      \[
        (n+1)a\leq nb\quad\Longrightarrow\quad a\leq b
        \qquad(a\in T(\alpha),\ n\geq1).
      \]
\item Writing $T(\alpha)^{\mathrm{ou}}$ for the order units,
      \[
       \pi\bigl(T(\alpha)^{\mathrm{ou}}\bigr)
       =\{h\in H(\alpha):
              \widehat h(\mu)>0\text{ for every }\mu\in\cM_\Gamma(X)\}.
      \]
\item The affine evaluation map associated with
      $Q=\cM_\Gamma(X)$ is good.
\end{enumerate}
\end{corollary}

\begin{proof}
For \textup{(i)}, since $b$ is an order unit,
 Lemma~\ref{lem:nm-order-unit-test} gives
$
   \mu(b)>0
$ for all $\mu\in\cM_\Gamma(X).$
If
$
   (n+1)a\leq nb,
$
then, for every $\mu\in\cM_\Gamma(X)$,
\[
   \mu(a)
   \leq
   \frac{n}{n+1}\mu(b)
   <
   \mu(b).
\]
Theorem~\ref{thm:main-comparison}\,\textup{(i)} therefore gives
$a\leq b$.

For \textup{(ii)}, let $c\in T(\alpha)^{\mathrm{ou}}$. By
Lemma~\ref{lem:nm-order-unit-test},
\[
   \widehat{\pi(c)}(\mu)=\mu(c)>0
   \qquad(\mu\in\cM_\Gamma(X)),
\]
which proves one inclusion.
Conversely, let $h=[f]_H\in H(\alpha)$ satisfy
\[
   \widehat h(\mu)>0
   \qquad(\mu\in\cM_\Gamma(X)).
\]
Write
\[
   f=f^+-f^-,
   \qquad
   a:=[f^-]_T,
   \qquad
   b:=[f^+]_T.
\]
Then
\[
   \mu(b)-\mu(a)
   =
   \int_X f\,d\mu
   =
   \widehat h(\mu)
   >0
\]
for every $\mu\in\cM_\Gamma(X)$.
By Theorem~\ref{thm:main-comparison}\,\textup{(i)},
$
   b=a+c
$
for some order unit $c\in T(\alpha)$. Applying $\pi$ gives
\[
   \pi(c)
   =
   \pi(b)-\pi(a)
   =
   [f^+]_H-[f^-]_H
   =
   [f]_H
   =
   h.
\]
Thus $h\in\pi(T(\alpha)^{\mathrm{ou}})$. Notice that no injectivity
of $\pi$ is used.

For \textup{(iii)}, let $A,B\subseteq X$ be clopen and suppose that
$
   \widehat A<\widehat B
 $ on $Q=\cM_\Gamma(X).$
By Theorem~\ref{thm:main-comparison}\,\textup{(i)},
$
   [A]\leq[B].
$
Hence, by the definition of the order on $T(\alpha)$, there is a
clopen set $D\subseteq B$ such that
$
   [D]=[A].
$
Every invariant measure therefore satisfies
$
   \mu(D)=\mu(A),
$
so
$
   \widehat D=\widehat A.
$
Thus the affine evaluation map associated with $Q$ is good.
\end{proof}

\begin{remark}\label{rem:nm-almost-unperforation-scope}
Corollary~\ref{cor:nm-direct-consequences}\,\textup{(i)} proves the
almost-unperforation implication when the target is an order unit.
In the nonminimal setting, however, a nonzero type $b$ may satisfy
$\mu(b)=0$ for some invariant probability measure. In that case,
$
   (n+1)a\leq nb
$
does not yield the strict inequalities
$\mu(a)<\mu(b)$ required to apply
Theorem~\ref{thm:main-comparison} \textup{(i)}.

For minimal actions every nonzero type is an order unit, so the
corollary recovers full almost unperforation. 
\end{remark}

\subsection{Algebraic comparison for minimal actions}
\label{subsec:regularity}
We now assume that $X$ is a Cantor space and that $\alpha$ is minimal.
Then $T(\alpha)$ is simple, so
Theorem~\ref{thm:main-three-two} can be iterated against the order unit
$u=[X]$ to obtain weak comparability.
The standard results collected in
Theorem~\ref{thm:known-type-results} then yield almost unperforation
and cancellation, proving
Theorem~\ref{thm:main-comparison}\,\textup{(ii)}.

\begin{proposition}\label{prop:weak-comparability}
Let $\alpha:\Gamma\acts X$ be a minimal Cantor action of a
countably infinite discrete amenable group. 
Then $T(\alpha)$ has weak comparability.
\end{proposition}

\begin{proof}
 Put
\[
   M=\Talpha
   \qquad\text{and}\qquad
   u=[X].
\]
It suffices to prove that for every $0\neq a\in M$, there exists $k\geq1$ such that
\[
   kb\leq u
   \quad\Longrightarrow\quad
   b\leq a
   \qquad(b\in M).
\]

Fix $0\neq a\in M$.  Since $M$ is simple \cite[Definition~2.3 and the following paragraph]{Melleray}, $a$ is an order
unit.  Hence there exists $m\geq1$ such that
\[
   u\leq ma.
\]

Choose $h\geq1$ such that
$
   m\leq2^h,
$
and set
$
   k:=3^h.
$
Let $b\in M$ satisfy
$
   kb\leq u.
$
Then
\begin{equation}\label{eq:weak-comparability-chain}
   3^hb
   =
   kb
   \leq u
   \leq ma
   \leq2^ha.
\end{equation}
Applying Theorem~\ref{thm:main-three-two} to \eqref{eq:weak-comparability-chain}
gives
\[
   3^{h-1}b\leq2^{h-1}a.
\]
Iterating this
implication yields
\[
   b\leq a.
\]
Thus, $M$ has weak comparability.
\end{proof}

Proposition~\ref{prop:weak-comparability}, together with
Theorem~\ref{thm:known-type-results}, proves
Theorem~\ref{thm:main-comparison}\,\textup{(ii)}.
\begin{proof}[Proof of
Theorem~\ref{thm:main-comparison} \textup{(ii)}]

By Proposition~\ref{prop:weak-comparability},
$\Talpha$ has weak comparability.  Hence
Theorem~\ref{thm:known-type-results} \textup{(i)} implies that
$\Talpha$ is almost unperforated.
Since $\Gamma$ is amenable and $X$ is nonempty and compact,
$
   \cM_\Gamma(X)\neq\varnothing.
$
Therefore Theorem~\ref{thm:known-type-results} \textup{(ii)} gives
cancellation of $\Talpha$. 
\end{proof}

 \subsection{Failure of the structural conclusions without minimality}
\label{subsec:nonminimal-counterexample}

The comparison conclusion of
Theorem~\ref{thm:main-comparison}\,\textup{(i)} holds
without minimality. In contrast, cancellation, the assertions in
\textup{(iii)}, and the exact full-group conclusion in
\textup{(iv)} do not admit an unconditional nonminimal extension.
The following example distinguishes each of these failures, even
for a free Cantor $\mathbb Z$-action.

The construction below follows the two-point compactification
modification suggested by Melleray in his corrections
\cite{MellerayCorrections} to the example in
\cite[Section~2.2]{Melleray}. We include the details and record
some further properties of this example: besides noncancellativity,
the coinvariant cone is not proper, and equality of clopen types
need not be realized by an element of the topological full group.

\begin{proposition}\label{prop:nm-structural-counterexample}
There exists a free nonminimal Cantor $\mathbb Z$-action $\alpha$
such that:
\begin{enumerate}[label=\textup{(\roman*)}]
\item $T(\alpha)$ is not cancellative and
      $\pi:T(\alpha)\to H(\alpha)^+$ is not injective;
\item $H(\alpha)^+$ is not a proper cone;
\item there are nonempty proper clopen sets $A,B\subseteq X$ with
      $[A]=[B]$ but $g(A)\neq B$ for every $g\in[[\alpha]]$.
\end{enumerate}
\end{proposition}

\begin{proof}
Let
\[
   Y:=\mathbb Z\cup\{-\infty,+\infty\}
\]
be the two-end compactification of $\mathbb Z$, let
$\tau(j)=j+1$ and $\tau(\pm\infty)=\pm\infty$, and let
$R(k)=k+1$ be the odometer on the $2$-adic integers $K$. Set
\[
   X:=Y\times K,
   \qquad
   T:=\tau\times R\qquad\text{and}\qquad\alpha(n):=T^n.
\]
Then $X$ is a Cantor space and $T$ is free. It is not minimal,
since the two end fibres $\{\pm\infty\}\times K$ are proper closed
invariant subsets.

For $j\in\mathbb Z$, put
\[
   C_j:=\{j\}\times K,
   \qquad
   D_+:=\bigl(\{0,1,2,\ldots\}\cup\{+\infty\}\bigr)\times K.
\]
Since
$
   T(D_+)=D_+\setminus C_0,
$
if $a=[D_+]$ and $c=[C_0]$, then
\[
   a=a+c,
   \qquad\text{but }
   c\neq0.
\]
Thus $T(\alpha)$ is not cancellative. Moreover,
\[
   \one_{C_0}
   =
   \one_{D_+}-\one_{D_+}\circ T^{-1},
\]
so $\pi(c)=0$, and hence $\pi$ is not injective.

To prove \textup{(ii)}, choose a nonempty proper clopen set
$W\subset K$ and put
\[
   h:=[\one_{\{0\}\times W}]_H.
\]
Since $[\one_{C_0}]_H=0$,
$
   -h=[\one_{\{0\}\times(K\setminus W)}]_H\in H(\alpha)^+.
$
It remains to show that $h\neq0$. If $h=0$, then, since for a
$\mathbb Z$-action every coboundary is of the form
$f-f\circ T^{-1}$,
\[
   \one_{\{0\}\times W}=f-f\circ T^{-1}
\]
for some $f\in C(X,\mathbb Z)$. Restricting the coboundary identity
to the two end fibres gives
\[
   f(\pm\infty,k)
   =
   f(\pm\infty,R^{-1}k)
   \qquad(k\in K),
\]
since $\{0\}\times W$ is disjoint from
$\{\pm\infty\}\times K$. Thus the functions
$
   f_\pm(k):=f(\pm\infty,k)
$
are $R$-invariant. Since the odometer $R$ is minimal and
$f_\pm$ are continuous, they are constant. Write
$
   f_+\equiv a_+,
$ and $
   f_-\equiv a_-,
$
for some $a_+,a_-\in\mathbb Z$.

Since $f$ is continuous and integer-valued, the sets
$
   \{x\in X:f(x)=a_+\},
 $ and $   \{x\in X:f(x)=a_-\}
$
are clopen neighborhoods of
$\{+\infty\}\times K$ and $\{-\infty\}\times K$, respectively.
By compactness of $K$, there exists $N\geq1$ such that
\[
   f(j,k)=a_+
   \quad(j\geq N,\ k\in K)
\qquad\text{and}\qquad
   f(j,k)=a_-
   \quad(j\leq -N,\ k\in K).
\]

Fix $k\in K$ and sum the coboundary identity along the orbit segment
$
  \{T^j(0,k)\}_{j=-N}^ N.
$
Since
$
   T^j(0,k)=(j,R^jk),
$
the orbit segment meets $\{0\}\times W$ only when $j=0$.
Hence
\[
   \sum_{j=-N}^{N}
   \one_{\{0\}\times W}\bigl(T^j(0,k)\bigr)
   =
   \one_W(k).
\]
On the other hand,
\begin{align*}
\sum_{j=-N}^{N}
\bigl(f-f\circ T^{-1}\bigr)\bigl(T^j(0,k)\bigr)
&=
\sum_{j=-N}^{N}
\Bigl(
   f\bigl(T^j(0,k)\bigr)
   -
   f\bigl(T^{j-1}(0,k)\bigr)
\Bigr)\\
&=
f\bigl(T^N(0,k)\bigr)
-
f\bigl(T^{-N-1}(0,k)\bigr)\\
&=
a_+-a_-.
\end{align*}
Therefore
\[
   \one_W(k)=a_+-a_-
   \qquad(k\in K).
\]
The right-hand side is independent of $k$, so $\one_W$ is constant.
This is impossible, as $W$ is nonempty and proper. Hence
$h\neq0$, and therefore
\[
   0\neq h\in H(\alpha)^+\cap(-H(\alpha)^+).
\]
So $H(\alpha)^+$ is not a proper cone.

Finally, let
\[
   A:=X\setminus C_0,
   \qquad
   B:=X\setminus(C_0\cup C_1).
\]
Define $\varphi:A\to B$ by
\[
   \varphi(x)=
   \begin{cases}
      T(x),
      &x\in
      \bigl(\{1,2,\ldots\}\cup\{+\infty\}\bigr)\times K,\\
      x,
      &x\in
      \bigl(\{\ldots,-2,-1\}\cup\{-\infty\}\bigr)\times K.
   \end{cases}
\]
Then $\varphi$ is a finite clopen equidecomposition from $A$ onto
$B$, and hence $[A]=[B]$. Suppose, towards a contradiction, that there exists
$g\in[[\alpha]]$ with $g(A)=B$. Since $g$ is a homeomorphism,
\[
   g(C_0)
   =
   g(X\setminus A)
   =
   X\setminus B
   =
   C_0\cup C_1.
\]
Every element of $[[\alpha]]$ preserves each $\mathbb Z$-orbit
setwise. Fix $k\in K$ and let
\[
   \mathcal O_k:=\{T^n(0,k):n\in\mathbb Z\}.
\]
Then
\[
   |\mathcal O_k\cap C_0|=1,
   \qquad
   |\mathcal O_k\cap(C_0\cup C_1)|=2.
\]
But the restriction of $g$ to $\mathcal O_k$ is a bijection, so it
cannot map $\mathcal O_k\cap C_0$ onto
$\mathcal O_k\cap(C_0\cup C_1)$. This contradiction proves that no
$g\in[[\alpha]]$ satisfies $g(A)=B$.
\end{proof}

\section{Coinvariants and affine evaluation}
\label{sec:consequences}
Throughout this section, $\alpha:\Gamma\acts X$ is a minimal
Cantor action of a countably infinite discrete amenable group, and
$
   Q:=\cM_\Gamma(X).
$

In this section, we begin by
using the cancellation obtained in Theorem \ref{thm:main-comparison}~\textup{(ii)} to prove the
remaining Theorem \ref{thm:main-comparison}~\textup{(iii)--(iv)}. 
We next use the comparison conclusion of
Theorem~\ref{thm:main-comparison}~\textup{(i)} to prove
Theorem~\ref{thm:main-aem}. It gives goodness of the affine
evaluation map associated with $Q$, yields a minimal homeomorphism
of $X$ with exactly the same invariant probability measures as the
original action, and identifies the ordered quotient
$
   H(\alpha)/\Inf(H(\alpha))
$
with the ordered group $G_Q$ determined by affine evaluation.

Having established Theorem~\ref{thm:main-aem}, we then turn to the
information lost in passing from clopen equidecomposition to affine
evaluation. We show that this loss is measured precisely by the
infinitesimal subgroup $\Inf(H(\alpha))$. In particular,
infinitesimals describe the obstruction to recovering clopen
equidecomposability from invariant measures alone, and they
parametrize the $[[\alpha]]$-orbits inside each nontrivial fibre of
the affine evaluation map.

\subsection{Proof of Theorem \ref{thm:main-comparison} (iii)-(iv)}

We begin with Theorem~\ref{thm:main-comparison} \textup{(iii)}.
The cancellation of $T(\alpha)$ established in
Theorem~\ref{thm:main-comparison} \textup{(ii)} turns the canonical
surjection onto $H(\alpha)^+$ into an isomorphism.

\begin{proof}[Proof of Theorem~\ref{thm:main-comparison} \textup{(iii)}]
Let
\[
   \pi:\Talpha\to\Halpha^+,
   \qquad
   \pi([f]_T)=[f]_H
\]
be the canonical map.  It is additive and
surjective.  By Theorem~\ref{thm:main-comparison} \textup{(ii)}, $\Talpha$ is
cancellative, and hence
Proposition \ref{prop:known-coinvariant} gives injectivity of $\pi$.
Therefore $\pi$ is a monoid isomorphism. 

Since $T(\alpha)$ is conical, $H(\alpha)^+$ is a proper cone.
Indeed, if $h,-h\in H(\alpha)^+$, choose $a,b\in T(\alpha)$
with
\[
   \pi(a)=h,\qquad \pi(b)=-h.
\]
Then $\pi(a+b)=0$, so $a+b=0$.  Conicality gives
$a=b=0$, and hence $h=0$.

Thus $H(\alpha)^+$ defines a partial order on $H(\alpha)$.
Moreover, for $a,b\in T(\alpha)$,
\begin{align*}
   a\leq b
   &\Longleftrightarrow
   b=a+c\text{ for some }c\in T(\alpha)\\
   &\Longleftrightarrow
   \pi(b)-\pi(a)=\pi(c)\in H(\alpha)^+\\
   &\Longleftrightarrow
   \pi(a)\leq\pi(b).
\end{align*}
Hence $\pi$ is an isomorphism of ordered monoids.
\end{proof}

The identification in
Theorem~\ref{thm:main-comparison} \textup{(iii)}, together with the
strict comparison in \textup{(i)}, gives a more concrete
description of the positive cone of $H(\alpha)$ and of the order
intervals determined by clopen sets. We record these consequences
for later use.

\begin{lemma}\label{lem:coinvariant-order-scale}
The canonical positive cone of the coinvariant group satisfies
\begin{equation}\label{eq:coinvariant-strict-cone}
   \Halpha^+
   =
   \{0\}\cup
   \{h\in\Halpha:
      \widehat h(\mu)>0\text{ for every }\mu\in Q\}.
\end{equation}
Moreover, for every clopen $U\subseteq X$,
\begin{equation}\label{eq:coinvariant-relative-scale}
   [0,[\one_U]_H]_{\Halpha}
   =
   \{[\one_A]_H:A\subseteq U\text{ is clopen}\}.
\end{equation}
\end{lemma}

\begin{proof}
Let $0\neq h\in\Halpha^+$.  Write
   $h=[f]_H$ for some $f\in C(X,\mathbb Z_{\geq0})$.
Then $f\not\equiv0$.  Since every measure in $Q$ has full support,
\[
   \widehat h(\mu)
   =
   \int_X f\,d\mu
   >0
   \qquad(\mu\in Q).
\]

Conversely, let $h=[f]_H\in\Halpha$ satisfy
$
   \widehat h(\mu)>0$ for all $\mu\in Q.$
Write
$
   f=f^+-f^-,
$
and put
$
   a=[f^-]_T,
 $ and $
   b=[f^+]_T.
$
Then, for every $\mu\in Q$, by \eqref{eq:formu}
\[
   \mu(b)-\mu(a)
   =
   \int_X f\,d\mu
   =
   \widehat h(\mu)
   >0.
\]
Hence Theorem~\ref{thm:main-comparison} \textup{(i)} gives
$
   a\leq b.
$
Write $b=a+c$ for some $c\in\Talpha$.  Using
Theorem~\ref{thm:main-comparison} \textup{(iii)},
\[
   h
   =
   [f^+]_H-[f^-]_H
   =
   \pi(b)-\pi(a)
   =
   \pi(c)
   \in\Halpha^+.
\]
This proves \eqref{eq:coinvariant-strict-cone}.

Now let
$
   0\leq h\leq[\one_U]_H.
$
By surjectivity of
$
   \pi:\Talpha\to\Halpha^+,
$
choose $a,c\in\Talpha$ such that
\[
   \pi(a)=h
   \qquad\text{and}\qquad
   \pi(c)=[\one_U]_H-h.
\]
Then
\[
   \pi(a+c)
   =
   [\one_U]_H
   =
   \pi([U]).
\]
By injectivity of $\pi$,
\[
   a+c=[U].
\]
Therefore, there is a clopen set $A\subseteq U$ with
$[A]=a$.  Hence
\[
   h=\pi(a)=[\one_A]_H.
\]

Conversely, if $A\subseteq U$ is clopen, then
$
   [\one_A]_H\in\Halpha^+
$
and
\[
   [\one_U]_H-[\one_A]_H
   =
   [\one_{U\setminus A}]_H
   \in\Halpha^+.
\]
Hence
\[
   0\leq[\one_A]_H\leq[\one_U]_H,
\]
which proves \eqref{eq:coinvariant-relative-scale}.
\end{proof}

We next prove Theorem~\ref{thm:main-comparison} \textup{(iv)}.
The key point is that cancellation allows an equidecomposition
between two clopen sets to be completed by an equidecomposition
between their complements, producing an element of the topological
full group.

\begin{proof}[Proof of Theorem~\ref{thm:main-comparison} \textup{(iv)}]
Let $A,B\subseteq X$ be clopen and suppose
$
   [A]=[B]
 $ in $\Talpha.$
Choose a finite clopen piecewise-$\Gamma$ bijection
$
   \varphi:A\to B
$
witnessing this equality.  Since
\[
   [A]+[X\setminus A]
   =
   [B]+[X\setminus B]
   =
   [X],
\]
cancellation gives
\[
   [X\setminus A]=[X\setminus B].
\]
Choose a finite clopen piecewise-$\Gamma$ bijection
$
   \psi:X\setminus A\to X\setminus B.
$
Then
$
   g:=\varphi\sqcup\psi
$
is a homeomorphism of $X$, belongs to $[[\alpha]]$, and satisfies
$g(A)=B$.
\end{proof}

Consequently, for clopen $A,B\subseteq X$,
\begin{equation}\label{eq:pair-equidecomposition}
   [\one_A]_H=[\one_B]_H
   \Longleftrightarrow
   A\sim_\Gamma B
   \Longleftrightarrow
   g(A)=B\text{ for some }g\in[[\alpha]].
\end{equation}
The first equivalence follows from
Theorem~\ref{thm:main-comparison} \textup{(iii)}, and the second from
\textup{(iv)}.

\subsection{Affine evaluation and the ordered quotient}

We now prove Theorem~\ref{thm:main-aem}.  Comparison gives goodness
of the affine evaluation map, and Proposition \ref{prop:known-realization} then
gives the required minimal homeomorphism.  The ordered quotient is
identified using Lemma~\ref{lem:coinvariant-order-scale}.

\begin{proof}[Proof of Theorem~\ref{thm:main-aem} \textup{(i)--(ii)}]
We only need to prove (i), as (ii) is a corollary of (i) and  Proposition \ref{prop:known-realization}.

Let $A,B\subseteq X$ be clopen and suppose
$
   \mu(A)<\mu(B)
$ for all $\mu\in Q.$
If $A=\varnothing$, take $D=\varnothing$.  Otherwise, dynamical
comparison gives a clopen partition
$
   A=\bigsqcup_{i=1}^n A_i
$
and elements $\gamma_1,\ldots,\gamma_n\in\Gamma$ such that the sets
$\gamma_iA_i$ are pairwise disjoint subsets of $B$.  Put
\[
   D:=\bigsqcup_{i=1}^n\gamma_iA_i.
\]
Then $D\subseteq B$ is clopen and, for every $\mu\in Q$,
\[
   \mu(D)
   =
   \sum_{i=1}^n\mu(\gamma_iA_i)
   =
   \sum_{i=1}^n\mu(A_i)
   =
   \mu(A).
\]
Thus the affine evaluation map associated with $Q$ is good.
\end{proof}

We now turn to the
ordered quotient in  \textup{(iii)}.  The identity on
$C(X,\mathbb Z)$ induces a natural map from $H(\alpha)$ onto $G_Q$, and
the point is to identify its kernel with $\Inf(H(\alpha))$ and then
use Lemma~\ref{lem:coinvariant-order-scale} to identify the induced
order and clopen scale.

 \begin{proof}[Proof of Theorem~\ref{thm:main-aem} \textup{(iii)}]
Put
$
   I:=\Inf(\Halpha),
$
let
$
   q:\Halpha\to\Halpha/I
$
be the quotient map, and write
$
   \bar u:=q([\one_X]_H).
$

Every coboundary $f-\gamma\cdot f$ integrates to zero against every
measure in $Q$.  Hence
$
   B_\alpha\subseteq N_Q,
$
and the identity on $C(X,\mathbb Z)$ induces a surjective
homomorphism
\[
   \Phi:\Halpha\to G_Q,
   \qquad
   \Phi([f]_H)=[f]_Q.
\]
Its kernel is
\begin{equation}\label{eq:kernel}
   \ker\Phi
   =
   N_Q/B_\alpha
   =
   I.    
\end{equation}
Thus $\Phi$ induces a group isomorphism
\[
   \overline\Phi:\Halpha/I\longrightarrow G_Q.
\]

We claim that this is an order isomorphism.  By
Lemma~\ref{lem:coinvariant-order-scale}, a nonzero
$h\in\Halpha$ belongs to $\Halpha^+$ exactly when
$
   \widehat h(\mu)>0
$ for all $\mu\in Q.
$
This is precisely the definition of the nonzero elements of $G_Q^+$.
Hence
\[
   \Phi(\Halpha^+)=G_Q^+,
\]
so $\overline\Phi$ identifies the quotient cone
$q(\Halpha^+)$ with $G_Q^+$.  It also sends $\bar u$ to
$[\one_X]_Q$.

By Theorem \ref{thm:main-aem} \textup{(i)}, the affine
evaluation map associated with $Q$ is good.   Proposition~\ref{prop:known-aem-structure} therefore shows
that $G_Q$ is a countable simple dimension group, that its normalized
state space is canonically affinely homeomorphic to $Q$, and that
\[
   [0,[\one_X]_Q]
   =
   \{[\one_A]_Q:A\in\Clop(X)\}.
\]

Since
\[
   \overline\Phi\bigl(q([f]_H)\bigr)=[f]_Q,
   \qquad
   \overline\Phi(\bar u)=[\one_X]_Q,
\]
the order isomorphism $\overline\Phi$ carries the interval
$[0,\bar u]$ onto $[0,[\one_X]_Q]$.  Hence the clopen-scale
identity above becomes
\[
   [0,\bar u]
   =
   \{q([\one_A]_H):A\in\Clop(X)\}.
\]
Likewise, the normalized state on $G_Q$ corresponding to
$\mu\in Q$
pulls back through $\overline\Phi$ to the normalized state on
$\Halpha/I$ given by
\[
   q([f]_H)\longmapsto\int_X f\,d\mu.
\]
Thus the normalized state space of $(\Halpha/I,\bar u)$ is
canonically affinely homeomorphic to $Q$, and the quotient has the
full clopen scale.  This proves \textup{(iii)}.
\end{proof}

We next identify the subgroup measuring the difference between
coinvariant relations and equality of affine evaluations. We retain the notation
\[
   I=\Inf(\Halpha),\qquad
   q:\Halpha\to\Halpha/I,\qquad
   \bar u=q([\one_X]_H),
\]
as well as the maps $\Phi$ and $\overline\Phi$, from the preceding
proof.

\begin{corollary}\label{cor:infinitesimal-relations}
The relation subgroups satisfy
\[
   B_\alpha\subseteq J_Q=N_Q,
   \qquad
   \Inf(\Halpha)=N_Q/B_\alpha.
\]
Consequently,
\begin{equation}\label{eq:coinvariant-exact-sequence}
   0\longrightarrow\Inf(\Halpha)
   \longrightarrow\Halpha
   \xrightarrow{\ \Phi\ }G_Q
   \longrightarrow0,
   \qquad
   \Phi([f]_H)=[f]_Q,
\end{equation}
is exact\footnote{A sequence of homomorphisms is exact if the image of each
map equals the kernel of the next.}, and $\Inf(\Halpha)$ is generated by the classes
\[
   [\one_A-\one_B]_H
   \qquad(A\sim_Q B).
\]
\end{corollary}

\begin{proof}
Let $f\in C(X,\mathbb Z)$ and $\gamma\in\Gamma$.
Decomposing the finite range of $f$ shows that
$f-\gamma\cdot f$ is an integer linear combination of functions
of the form
$
   \one_A-\one_{\gamma A}.
$
Since
$
   \mu(A)=\mu(\gamma A)
$ for all $\mu\in Q,$
each such difference belongs to $J_Q$.  Hence
$
   B_\alpha\subseteq J_Q.
$
By Proposition~\ref{prop:known-aem-structure} \textup{(iii)},
goodness gives
$
   J_Q=N_Q.
$
The identity
$
   \Inf(\Halpha)=N_Q/B_\alpha
$
is from \eqref{eq:kernel}.
The exact sequence follows immediately, and the description of the
generators follows from the definition of $J_Q$.
\end{proof}

\subsection{Exact equidecomposition and infinitesimals}
\label{subsec:exact-equidecomposition}
The preceding exact sequence identifies the obstruction to replacing
equality of affine evaluation functions by an actual clopen
equidecomposition.  If $A,B\subseteq X$ are clopen and
$A\sim_Q B$, then
\[
   [\one_A-\one_B]_H\in\Inf(\Halpha),
\]
while \eqref{eq:pair-equidecomposition} shows that
$A\sim_\Gamma B$ exactly when this infinitesimal is zero.

We now characterize the vanishing of this obstruction in several
equivalent ways, in terms of the type semigroup, clopen
equidecomposition, the topological full group, and the relation
subgroups.

\begin{theorem}[Theorem~\ref{thm:main-infinitesimals} \textup{(i)}]\label{thm:exact-lifting}
The following conditions are equivalent:
\begin{enumerate}[label=\textup{(\roman*)}]
\item $\Inf(\Halpha)=0$;
\item the evaluation map
      \[
         \rho_Q:\Talpha\to\Aff(Q),
         \qquad
         \rho_Q(a)(\mu)=\mu(a),
      \]
      is injective;
\item for all clopen $A,B\subseteq X$,
      $A\sim_Q B$ implies $A\sim_\Gamma B$, equivalently
      $g(A)=B$ for some $g\in[[\alpha]]$;
\item
      $
         \overline{[[\alpha]]}=\mathcal H_Q,
      $
      where the closure is taken in $\Homeo(X)$ 
      with respect to the topology of uniform
      convergence of maps and inverses;
\item $
   B_\alpha=N_Q.
$
\end{enumerate}
\end{theorem}

\begin{proof}
By Corollary~\ref{cor:infinitesimal-relations},
 (i) $\Leftrightarrow$ (v).

Assume \textup{(i)} holds and 
$
   \rho_Q(a)=\rho_Q(b).
$
Under the isomorphism of
Theorem~\ref{thm:main-comparison} \textup{(iii)}, $\pi(a)-\pi(b)\in\Inf(H(\alpha))$.  Hence $\pi(a)=\pi(b)$, and therefore $a=b$.  Thus
\textup{(ii)} holds.

If~\textup{(ii)} holds and $A\sim_Q B$, then
\[
   \rho_Q([A])=\rho_Q([B]),
\]
so $[A]=[B]$.  By
\eqref{eq:pair-equidecomposition},
$A\sim_\Gamma B$, proving~\textup{(iii)}.

Assume~\textup{(iii)}.  Every generator
$
   \one_A-\one_B
$
of $J_Q$ has $A\sim_Q B$, and hence $A\sim_\Gamma B$.
Therefore
\[
   J_Q\subseteq B_\alpha.
\]
By Corollary \ref{cor:infinitesimal-relations}, one has
$
   B_\alpha\subseteq J_Q=N_Q,
$
and hence $B_\alpha=N_Q$. So \textup{(i)} holds.

Every element of $[[\alpha]]$ preserves every measure in $Q$, and
$\mathcal H_Q$ is closed in $\Homeo(X)$.  Hence
\[
   \overline{[[\alpha]]}\subseteq\mathcal H_Q.
\]
Assume~\textup{(iii)} and fix $h\in\mathcal H_Q$.
Given $\varepsilon>0$, choose a finite clopen partition
$
   X=A_1\sqcup\cdots\sqcup A_k
$
such that
$
   \diam(A_i)<\varepsilon,
$ and  $
   \diam(h(A_i))<\varepsilon$
for every $i$.  For each $i$,
$
   A_i\sim_Q h(A_i),
$
so~\textup{(iii)} gives a finite clopen piecewise-$\Gamma$
bijection
$
   A_i\to h(A_i).
$
Gluing these maps gives $g\in[[\alpha]]$ satisfying
\[
   g(A_i)=h(A_i)
   \qquad(1\leq i\leq k).
\]
For $x\in A_i$, both $g(x)$ and $h(x)$ belong to $h(A_i)$, and 
hence
\[
   d(g(x),h(x))<\varepsilon.
\]
Likewise, for $y\in h(A_i)$, both $g^{-1}(y)$ and
$h^{-1}(y)$ belong to $A_i$, so
\[
   d(g^{-1}(y),h^{-1}(y))<\varepsilon.
\]
As $\varepsilon>0$ is arbitrary, it follows that \textup{(iv)} holds.

Finally, assume~\textup{(iv)} and let $A\sim_Q B$.
Proposition~\ref{prop:known-aem-structure} \textup{(iv)} gives
$h\in\mathcal H_Q$ with
$
   h(A)=B.
$
Since $A$ and $B$ are clopen, the set of homeomorphisms carrying
$A$ onto $B$ is an open neighborhood of $h$ in
$\Homeo(X)$.  By~\textup{(iv)}, this neighborhood meets
$[[\alpha]]$.  Hence some $g\in[[\alpha]]$ satisfies
$g(A)=B$, proving~\textup{(iii)}.
\end{proof}

Theorem~\ref{thm:exact-lifting} characterizes exactly when affine
evaluation loses no information about clopen equidecomposition.  We
now describe what happens when $\Inf(\Halpha)$ is nonzero.  For every
nonempty proper clopen set $A$, we show that the fibre of the affine
evaluation map through $A$ decomposes into $[[\alpha]]$-orbits
canonically indexed by $\Inf(\Halpha)$.

\begin{theorem}[Theorem~\ref{thm:main-infinitesimals} \textup{(ii)}]
\label{thm:profile-orbits}
Fix a clopen set
$\varnothing\neq A\subsetneq X$
and let
\[
   \mathscr F_A
   :=
   \{B\in\Clop(X):\widehat B=\widehat A\},
\]
the fibre of the affine evaluation map through $A$.  Then the map
\begin{equation}\label{eq:profile-orbit-bijection}
   \mathscr F_A/[[\alpha]]
   \to
   \Inf(\Halpha),
   \qquad
   [[\alpha]]\cdot B
   \mapsto
   [\one_B-\one_A]_H,
\end{equation}
is a bijection.  Here $\mathscr F_A/[[\alpha]]$ denotes the set of
$[[\alpha]]$-orbits in $\mathscr F_A$.
\end{theorem}

\begin{proof}
Since every element of $[[\alpha]]$ preserves every measure in
$Q$, the group $[[\alpha]]$ acts on $\mathscr F_A$.  For
$B\in\mathscr F_A$,
\[
   [\one_B-\one_A]_H\in\Inf(\Halpha),
\]
and this class is constant on each $[[\alpha]]$-orbit.  Thus the map
in \eqref{eq:profile-orbit-bijection} is well defined.

Suppose $B,C\in\mathscr F_A$ have the same image, i.e.
$
   [\one_B-\one_A]_H
   =
   [\one_C-\one_A]_H,
$
and hence
$
   [\one_B]_H=[\one_C]_H.
$
By \eqref{eq:pair-equidecomposition}, there is
$g\in[[\alpha]]$ such that $g(B)=C$.  The map is therefore
injective.

For surjectivity, let $t\in\Inf(\Halpha)$ and put
\[
   \eta:=[\one_A]_H+t.
\]
Since $A$ is nonempty and proper and every measure in $Q$ has
full support,
\[
   \widehat\eta(\mu)=\mu(A)>0,
   \qquad
   \widehat{[\one_X]_H-\eta\;}(\mu)=1-\mu(A)>0
   \qquad(\mu\in Q).
\]
Lemma~\ref{lem:coinvariant-order-scale} therefore gives
\[
   0\leq\eta\leq[\one_X]_H.
\]
By \eqref{eq:coinvariant-relative-scale}, applied with $U=X$, there
is a clopen set $B\subseteq X$ such that
\[
   [\one_B]_H=\eta.
\]
Since $t$ is infinitesimal,
$
   \widehat B=\widehat A,
$
so $B\in\mathscr F_A$, and
\[
   [\one_B-\one_A]_H=t.
\]
Thus the map is surjective.
\end{proof}

\begin{remark}
Proposition~\ref{prop:known-aem-structure} \textup{(iv)} shows that
$\mathscr F_A$ is a single $\mathcal H_Q$-orbit.  Theorem~
\ref{thm:profile-orbits} describes exactly how this orbit splits into
$[[\alpha]]$-orbits: the orbit set is naturally parametrized by
$\Inf(\Halpha)$.  In particular, every nontrivial fibre of the
affine evaluation map contains exactly
$
   |\Inf(\Halpha)|
$ 
$[[\alpha]]$-orbits.

The assumption $\varnothing\neq A\subsetneq X$ is essential for
this statement.  Indeed, full support implies that the fibre through
$\varnothing$ consists only of $\varnothing$, while the fibre
through $X$ consists only of $X$.
\end{remark}

\section{Almost finiteness beyond Cantor spaces}
\label{sec:almost-finite}
Throughout this section, $\Gamma$ is a countably infinite discrete
amenable group, and all spaces are nonempty and compact metrizable.

We now derive the almost-finiteness consequences of
Theorem~\ref{thm:main-comparison} \textup{(i)}.
For a free zero-dimensional action, dynamical comparison together
with the small boundary property implies almost finiteness by
\cite[Theorem~6.1]{KerrSzabo}. Thus every free zero-dimensional
action of $\Gamma$ is almost finite.
Kerr and Szab\'o's reduction theorem
\cite[Theorem~7.6]{KerrSzabo} then yields almost finiteness for every
free action of $\Gamma$ with the topological small boundary property,
and their Corollary~7.7 gives the finite-dimensional case.
This proves Theorem~\ref{thm:main-almost-finite} and settles
Naryshkin's almost-finiteness conjecture.
Finally, for minimal actions, we
combine almost finiteness with the relevant $C^*$-algebraic results
to prove Theorem~\ref{thm:main-crossed-products}.

\subsection{The zero-dimensional case}
\label{subsec:cantor-almost-finite}
\begin{corollary}\label{cor:cantor-almost-finite}
Every free action $\alpha:\Gamma\acts Z$ on a nonempty compact
metrizable zero-dimensional space is almost finite. If the action
is also minimal, $C(Z)\rtimes_r\Gamma$ is $\mathcal Z$-stable.
\end{corollary}

\begin{proof}
By Theorem~\ref{thm:main-comparison} \textup{(i)}, the action has
clopen dynamical comparison, hence open-set comparison by
Lemma~\ref{lem:clopen-open-comparison}. A clopen basis gives SBP.
Theorem~\ref{thm:known-almost-finite} \textup{(i)} therefore
gives almost finiteness. If the action is minimal, apply
Theorem~\ref{thm:known-almost-finite} \textup{(ii)} for the
crossed-product conclusion.
\end{proof}

\subsection{The topological small boundary property and Naryshkin's conjecture}
\label{subsec:proof-main-almost-finite}

The zero-dimensional case established above provides exactly the
hypothesis needed for the reduction theorem of Kerr and Szab\'o.
We now apply their result to pass to actions with the TSBP, and then to finite-dimensional spaces.
This completes the proof of Theorem~\ref{thm:main-almost-finite}.

\begin{proof}[Proof of Theorem~\ref{thm:main-almost-finite}]
By Corollary~\ref{cor:cantor-almost-finite}, every
free action of $\Gamma$ on a nonempty compact metrizable
zero-dimensional space is almost finite. Theorem~\ref{thm:ks-reduction}, therefore implies that every
free action of $\Gamma$ on a nonempty compact metrizable space with
TSBP is almost finite. The finite-dimensional assertion follows from
Proposition~\ref{prop:external-finite-dim} to obtain TSBP.
\end{proof}

\subsection{Regularity of the crossed products}
\label{subsec:crossed-product-regularity}

\begin{proof}[Proof of Theorem~\ref{thm:main-crossed-products}]
By Theorem~\ref{thm:main-almost-finite}, the action is almost finite.
Since it is also free and minimal,
Theorem~\ref{thm:known-almost-finite}\,\textup{(ii)} gives
$\mathcal Z$-stability of $A=C(X)\rtimes_r\Gamma$.

The crossed product $A$ is separable, simple, unital, and nuclear
by the standard properties of crossed products of free minimal
amenable actions; see \cite{BrownOzawa}.
Moreover, $X$ is infinite because $\Gamma$ is infinite and acts
freely, so the inclusion $C(X)\subseteq A$ shows that $A$ is
infinite-dimensional.
By \cite[Theorem~B]{CastillejosEvingtonTikuisisWhiteWinter},
\[
    \dim_{\mathrm{nuc}}(A)\leq 1.
\]

By \cite[Theorem~4.5]{RordamZStable}, $\mathcal Z$-stability implies
that $W(A)$ is almost unperforated. Since nuclearity implies
exactness, \cite[Corollary~4.6]{RordamZStable}, applied to matrix
algebras over $A$, also gives strict comparison of positive elements.
Thus all four regularity conditions hold.
The finite-dimensional case follows from
Proposition~\ref{prop:external-finite-dim}.
\end{proof}

\section*{Statements and Declarations}

\noindent\textbf{Acknowledgements.} ChatGPT was used to assist with language editing, exposition, and clarifying the presentation of some proof ideas and explanations. All mathematical content is entirely the authors' own, and the authors take full responsibility for the final manuscript.

\medskip

\noindent\textbf{Funding.}
This work was supported by the Liaoning Provincial Natural Science
Foundation Doctoral Research Start-up Project (Project No.~2026-BS-0032),
the Postdoctoral Fellowship Program and China Postdoctoral Science
Foundation under Grant Number BX20250067, and the China Postdoctoral
Science Foundation under Grant Number 2025M773074.

\medskip

\noindent\textbf{Competing Interests.}
The authors have no relevant financial or non-financial interests to disclose.


\begin{thebibliography}{99}

\bibitem{CastillejosEvingtonTikuisisWhiteWinter}
Jorge Castillejos, Samuel Evington, Aaron Tikuisis,
Stuart White, and Wilhelm Winter,
\emph{Nuclear dimension of simple $C^*$-algebras},
Invent. Math. \textbf{224} (2021), no.~1, 245--290.

\bibitem{AnantharamanDelaroche2002}
Claire Anantharaman-Delaroche,
\emph{Amenability and exactness for dynamical systems and their
$C^*$-algebras},
Trans. Amer. Math. Soc. \textbf{354} (2002), no.~10, 4153--4178.

\bibitem{AnantharamanDelarocheRenault}
Claire Anantharaman-Delaroche and Jean N. Renault,
\emph{Amenable groupoids},
Monographies de L'Enseignement Math{\'e}matique, vol.~36,
L'Enseignement Math{\'e}matique, Geneva, 2000.

\bibitem{BoldriniPrasad}
Paolo Boldrini and Akshara Prasad,
\emph{Topologically free minimal actions without dynamical comparison},
preprint, arXiv:2607.01896, 2026.

\bibitem{BrownOzawa}
Nathanial P. Brown and Narutaka Ozawa,
\emph{$C^*$-algebras and finite-dimensional approximations},
Graduate Studies in Mathematics, vol.~88,
American Mathematical Society, Providence, RI, 2008.

\bibitem{ConleyJacksonKerrMarksSewardTuckerDrob}
Clinton T. Conley, Steve C. Jackson, David Kerr, Andrew S. Marks,
Brandon Seward, and Robin D. Tucker-Drob,
\emph{F{\o}lner tilings for actions of amenable groups},
Math. Ann. \textbf{371} (2018), no.~1--2, 663--683.

\bibitem{DownarowiczZhang}
Tomasz Downarowicz and Guohua Zhang,
\emph{Symbolic extensions of amenable group actions and the comparison
property},
Mem. Amer. Math. Soc. \textbf{281} (2023), no.~1390, vi+95 pp.

\bibitem{ElekLippner}
G{\'a}bor Elek and G{\'a}bor Lippner,
\emph{Borel oracles. An analytical approach to constant-time algorithms},
Proc. Amer. Math. Soc. \textbf{138} (2010), no.~8, 2939--2947.

\bibitem{ElliottNiu}
George A. Elliott and Zhuang Niu,
\emph{The $C^*$-algebra of a minimal homeomorphism of zero mean
dimension},
Duke Math. J. \textbf{166} (2017), no.~18, 3569--3594.

\bibitem{GiordanoMatuiPutnamSkau2010}
Thierry Giordano, Hiroki Matui, Ian F. Putnam, and Christian F. Skau,
\emph{Orbit equivalence for Cantor minimal $\mathbb Z^d$-systems},
Invent. Math. \textbf{179} (2010), no.~1, 119--158.

\bibitem{GiordanoPutnamSkau}
Thierry Giordano, Ian F. Putnam, and Christian F. Skau,
\emph{Topological orbit equivalence and $C^*$-crossed products},
J. Reine Angew. Math. \textbf{469} (1995), 51--111.

\bibitem{Glasner}
Eli Glasner,
\emph{Affine evaluation maps, dimension groups, and fair measures},
preprint, arXiv:2608.18700v1, 2026.

\bibitem{GlasnerWeiss}
Eli Glasner and Benjamin Weiss,
\emph{Weak orbit equivalence of Cantor minimal systems},
Internat. J. Math. \textbf{6} (1995), no.~4, 559--579.

\bibitem{HermanPutnamSkau}
Richard H. Herman, Ian F. Putnam, and Christian F. Skau,
\emph{Ordered Bratteli diagrams, dimension groups and topological dynamics},
Internat. J. Math. \textbf{3} (1992), no.~6, 827--864.

\bibitem{IbarluciaMelleray}
Tom{\'a}s Ibarluc{\'i}a and Julien Melleray,
\emph{Dynamical simplices and minimal homeomorphisms},
Proc. Amer. Math. Soc. \textbf{145} (2017), no.~11, 4981--4994.

\bibitem{Kerr}
David Kerr,
\emph{Dimension, comparison, and almost finiteness},
J. Eur. Math. Soc. (JEMS) \textbf{22} (2020), no.~11, 3697--3745.

\bibitem{KerrNaryshkin}
David Kerr and Petr Naryshkin,
\emph{Elementary amenability and almost finiteness},
Compos. Math. \textbf{161} (2025), no.~12, 3321--3337.

\bibitem{KerrSzabo}
David Kerr and G{\'a}bor Szab{\'o},
\emph{Almost finiteness and the small boundary property},
Comm. Math. Phys. \textbf{374} (2020), no.~1, 1--31.

\bibitem{Koppelberg}
Sabine Koppelberg,
\emph{General theory of Boolean algebras},
in \emph{Handbook of Boolean Algebras}, vol.~1,
J. Donald Monk and Robert Bonnet (eds.),
North-Holland, Amsterdam, 1989.

\bibitem{Lindenstrauss1999}
Elon Lindenstrauss,
\emph{Mean dimension, small entropy factors and an embedding theorem},
Publ. Math. Inst. Hautes \'Etudes Sci. \textbf{89} (1999), 227--262.

\bibitem{Ma}
Xin Ma,
\emph{A generalized type semigroup and dynamical comparison},
Ergodic Theory Dynam. Systems \textbf{41} (2021), no.~7,
2148--2165.

\bibitem{Melleray}
Julien Melleray,
\emph{Clopen type semigroups of actions on $0$-dimensional compact spaces},
Groups Geom. Dyn. \textbf{19} (2025), no.~3, 957--987.

\bibitem{MellerayCorrections}
Julien Melleray,
\emph{Some minor corrections to the paper
``Clopen type semigroups of actions on $0$-dimensional compact spaces''},
author's correction note,
\url{https://math.univ-lyon1.fr/~melleray/clopen_comments.pdf}.

\bibitem{Naryshkin}
Petr Naryshkin,
\emph{Polynomial growth, comparison, and the small boundary property},
Adv. Math. \textbf{406} (2022), Paper No.~108519, 9 pp.

\bibitem{NaryshkinExtensions}
Petr Naryshkin,
\emph{Group extensions preserve almost finiteness},
J. Funct. Anal. \textbf{286} (2024), no.~7,
Paper No.~110348, 8 pp.

\bibitem{Renault1980}
Jean N. Renault,
\emph{A groupoid approach to $C^*$-algebras},
Lecture Notes in Mathematics, vol.~793,
Springer-Verlag, Berlin, 1980.

\bibitem{RenaultWilliams}
Jean N. Renault and Dana P. Williams,
\emph{Amenability of groupoids arising from partial semigroup actions
and topological higher rank graphs},
Trans. Amer. Math. Soc. \textbf{369} (2017), no.~4, 2255--2283.

\bibitem{RordamZStable}
Mikael R{\o}rdam,
\emph{The stable and the real rank of $\mathcal Z$-absorbing
$C^*$-algebras},
Internat. J. Math. \textbf{15} (2004), no.~10, 1065--1084.

\bibitem{RordamSierakowski}
Mikael R{\o}rdam and Adam Sierakowski,
\emph{Purely infinite $C^*$-algebras arising from crossed products},
Ergodic Theory Dynam. Systems \textbf{32} (2012), no.~1, 273--293.

\bibitem{SzaboRokhlin}
G{\'a}bor Szab{\'o},
\emph{The Rokhlin dimension of topological $\mathbb Z^m$-actions},
Proc. Lond. Math. Soc. (3) \textbf{110} (2015), no.~3, 673--694.

\bibitem{TomsWinter}
Andrew S. Toms and Wilhelm Winter,
\emph{Minimal dynamics and $K$-theoretic rigidity: Elliott's conjecture},
Geom. Funct. Anal. \textbf{23} (2013), no.~1, 467--481.

\bibitem{WinterZachariasNuclearDimension}
Wilhelm Winter and Joachim Zacharias,
\emph{The nuclear dimension of $C^*$-algebras},
Adv. Math. \textbf{224} (2010), no.~2, 461--498.

\end{thebibliography}
\end{document}